\documentclass[12pt]{amsart}
\usepackage{txfonts}
\usepackage{amscd,amssymb,mathabx,subfigure}
\usepackage[arrow,matrix,graph,frame,poly,arc,tips]{xy}
\usepackage{graphicx,calrsfs}
\usepackage[colorlinks,plainpages,backref,urlcolor=blue]{hyperref}
\usepackage{tikz}
\usetikzlibrary{decorations.markings}
\usetikzlibrary{shapes}
\usetikzlibrary{backgrounds}
\usetikzlibrary{calc}
\usepackage{mdwlist}
\usepackage{multirow}
\usepackage{enumerate}
\usepackage[shortlabels]{enumitem}
\usepackage{verbatim}
\usepackage[abs]{overpic}

\title[RFR$p$ groups]{Residually finite rationally $p$ groups} 

\author{Thomas Koberda}
\address{Department of Mathematics, University of Virginia, 
Charlottesville, VA 22904-4137, USA}
\email{thomas.koberda@gmail.com}
\urladdr{\href{http://faculty.virginia.edu/Koberda/}%
{http://faculty.virginia.edu/Koberda/}}

\author[Alexander I. Suciu]{Alexander I. Suciu}
\address{Department of Mathematics, Northeastern 
University, Boston, MA 02115, USA}
\email{a.suciu@northeastern.edu}
\urladdr{\href{http://web.northeastern.edu/suciu/}%
{http://web.northeastern.edu/suciu/}}

\usepackage{enumerate}
\makeatletter
\let\@@enum@org\@@enum@
\def\@@enum@[#1]{\@@enum@org[\normalfont #1]}
\makeatother

\newtheorem{thm}{Theorem}[section]
\newtheorem{lem}[thm]{Lemma}
\newtheorem{lemma}[thm]{Lemma}
\newtheorem{cor}[thm]{Corollary}
\newtheorem{prop}[thm]{Proposition}

\newtheorem{que}[thm]{Question}
\newtheorem{prob}[thm]{Problem}

\theoremstyle{definition}

\newtheorem{example}[thm]{Example}

\newtheorem{remark}[thm]{Remark}

\newcommand\R{\mathbb{R}}
\newcommand\Q{\mathbb{Q}}
\newcommand\Z{\mathbb{Z}}
\newcommand\C{\mathbb{C}}
\newcommand\CP{\mathbb{CP}}

\newcommand\bZ{\mathbb{Z}}
\newcommand\bH{\mathbb{H}}
\newcommand\bP{\mathbb{P}}
\newcommand\bC{\mathbb{C}}
\newcommand\bQ{\mathbb{Q}}

\newcommand\bR{\mathbb{R}}

\newcommand\mC{\mathcal{C}}
\newcommand\mL{\mathcal{L}}

\newcommand\A{\mathcal{A}}

\renewcommand\L{\mathcal{L}}
\renewcommand\P{\mathcal{P}}
\newcommand\mP{\mathcal{P}}
\newcommand\mX{\mathcal{X}}

\newcommand{\abs}[1]{\left| #1 \right|}

\newcommand\Aut{\operatorname{Aut}}

\newcommand\GL{\operatorname{GL}}

\newcommand\lk{\operatorname{Lk}}
\newcommand\rk{\operatorname{rk}}
\newcommand\rad{\operatorname{rad}}
\newcommand\st{\operatorname{St}}

\newcommand\PSL{\operatorname{PSL}}
\newcommand\Nil{\operatorname{Nil}}
\newcommand\Sol{\operatorname{Sol}}

\newcommand\cay{\operatorname{Cayley}}
\newcommand\TF{\operatorname{TFr}}
\newcommand\T{\operatorname{Tors}}

\newcommand{\ab}{\operatorname{ab}}

\newcommand\gam{\Gamma}

\newcommand\yh{\widehat}

\newcommand{\surj}{\twoheadrightarrow}

\newcounter{tmpc}

\begin{document}

\date{\today}



\begin{abstract}
In this article we develop the theory of residually finite rationally $p$ (RFR$p$) 
groups, where $p$ is a prime. We first prove a series of results about the structure 
of finitely generated RFR$p$ groups (either for a single prime $p$, or for infinitely 
many primes), including torsion-freeness, a Tits alternative, and a restriction on 
the BNS invariant.  Furthermore, we show that many groups which occur naturally 
in group theory, algebraic geometry, and in $3$-manifold topology enjoy this residual 
property.  We then prove a combination theorem for RFR$p$ groups, which we use 
to study the boundary manifolds of algebraic curves $\CP^2$ and in $\C^2$. 
We show that boundary manifolds of a large class of curves in $\C^2$ (which 
includes all line arrangements) have RFR$p$ fundamental groups, whereas 
boundary manifolds of curves in $\CP^2$ may fail to do so.  
\end{abstract}
\maketitle

\maketitle
\setcounter{tocdepth}{1}
\tableofcontents

\section{Introduction}
\label{sect:intro}

In this paper, we develop a group theoretic property called \emph{residually finite 
rationally $p$}.  We study the class of finitely generated groups with this property 
for its own sake, and we study this property among several classes of groups which 
occur in algebraic geometry and in $3$-manifold topology. Most notably, we show 
that this property is enjoyed by the boundary manifold of a curve arrangement with 
only type $A$ singularities in $\C^2$, but that the analogous property for boundary 
manifolds of curve arrangements in $\C\bP^2$ does not necessarily hold.

\subsection{The class of RFR$p$ groups}
\label{subsec:rfrp}
Let $p$ be a prime. A finitely generated group $G$ is called 
\emph{residually finite rationally $p$} of \emph{RFR$p$} if there 
exists a sequence of nested, finite index subgroups 
$\{G_i\}_{i\geq 1}$ of $G$ such that:
\begin{enumerate}
\item
$G=G_1$.
\item
The intersection of the groups $\{G_i\}$ is trivial.
\item
Each quotient $G_i/G_{i+1}$ is an elementary abelian $p$-group.
\item
Every element $g\in G_i\setminus G_{i+1}$ represents a nonzero 
class in $H_1(G_i,\Q)$.
\end{enumerate}

Let us define the \emph{RFR$p$ topology}\/ on $G$ as the weakest topology on $G$ for which 
 the standard RFR$p$ filtration 
of $G$ (cf.~Lemma \ref{l:normal}) consists of open sets. The group $G$ is RFR$p$ precisely 
when this topology is Hausdorff, 
or equivalently, the trivial group is a closed subgroup.

Many finitely generated groups which occur naturally in 
geometric group theory and in topology are RFR$p$.  
For instance, we have the following result, which is a 
combination of Propositions \ref{prop:free rfrp}, \ref{prop:free surf}, 
and \ref{prop:raag rfrp}.

\begin{prop}
\label{prop:ex rfrp}
The following groups are RFR$p$, for all primes $p$:
\begin{enumerate}
\item \label{free}
Finitely generated free groups.
\item \label{surf}
Closed, orientable surface groups.
\item \label{raag}
Right-angled Artin groups.
\end{enumerate}
\end{prop}

\subsection{Properties of RFR$p$ groups}
\label{subsec:props rfrs}
The groups we study here enjoy many useful properties. We present 
some of these properties in the following theorem, which summarizes 
Proposition \ref{prop:rprp props} and Theorems \ref{thm:rtfpoly} and 
\ref{thm:separable}, as well as Corollary \ref{c:rfrpfg} and 
Theorems \ref{t:rfrplarge} and \ref{t:beauville}.

\begin{thm}
\label{t:rfrpsummary}
Let $G$ be a finitely generated group which is RFR$p$ for some 
prime $p$.  Then:
\begin{enumerate}
\item \label{res1}
$G$ is residually finite.  In particular, if $G$ is finitely presented then $G$ 
has a solvable word problem.
\item \label{res2}
$G$ is torsion-free.
\item \label{res3}
$G$ is residually torsion-free polycyclic.
\item \label{res8}
For each $n$, the maximal abelian subgroups of $G$ of rank $n$ are separable.
\setcounter{tmpc}{\theenumi}
\end{enumerate}
If, moreover, the group $G$ is RFR$p$ for infinitely many primes, 
finitely presented, and nonabelian, then:
\begin{enumerate}
\setcounter{enumi}{\thetmpc}
\item \label{res4}
$G$ is large,  i.e., $G$ virtually surjects 
to a nonabelian free group.
\item \label{res5}
The maximal $k$-step solvable quotients $G/G^k$ are not finitely presented, 
for any $k\geq 2$.
\item \label{res6}
The derived subgroup $G'$ is not finitely generated.
\item \label{res7}
The complement of the BNS invariant $\Sigma^1(G)$ is not empty. 
\item \label{res9}
The group $G$ is bi-orderable.
\end{enumerate}
\end{thm}

We remark that questions around the finite presentability of 
metabelian quotients of finitely presented groups (such as
item~\ref{res5} above) have been recently investigated by 
R.~Strebel~\cite{StrebelPreprint}.

\subsection{A combination theorem}
\label{subsec:props}
Our main result about RFR$p$ groups is a combination theorem which 
allows us to construct many new RFR$p$ groups from old ones:

\begin{thm}[Theorem \ref{thm:combination}]
\label{t:comb}
Fix a prime $p$. Let $G=G_{\gam}$ be a finite graph of finitely generated 
groups with vertex groups $\{G_v\}_{v\in V(\gam)}$ and groups 
$\{G_e\}_{e\in E(\gam)}$ satisfying the following conditions:
\begin{enumerate}
\item
For each $v\in V(\gam)$, the group $G_v$ is RFR$p$.
\item
For each $v\in V(\gam)$, the RFR$p$ topology on $G$ induces 
the RFR$p$ topology on $G_v$.
\item
For each $e\in E(\gam)$ and each $v\in e$, we have that the image 
of $G_e$ in $G_v$ given by the graph of groups structure of $G$ is 
closed in the RFR$p$ topology on $G_v$.
\end{enumerate}
Then $G$ is RFR$p$.
\end{thm}

The reader is directed to Section \ref{sect:graph of groups} for the relevant 
technical definitions.

\subsection{$3$-manifold topology}
\label{subsec:3mfd}
The RFR$p$ property also produces a new invariant of $3$-manifold groups 
which is finer than previously studied residual properties enjoyed by $3$-manifolds:

\begin{thm}
\label{thm:geometric}
Let $G=\pi_1(M)$ be a geometric $3$-manifold group, possibly with toroidal boundary. 
Then there is a finite index subgroup $G_0<G$ which is RFR$p$ for every prime $p$ 
if and only if $M$ admits one of the following geometries: 
$\{S^3,S^2\times\R,\R^3,\bH^2\times\R,\bH^3\}$. Otherwise, no finite index 
subgroup of $G$ is RFR$p$ for any prime.
\end{thm}

We remark that Theorem \ref{thm:geometric} relies on Agol's resolution of 
the virtual Haken conjecture~\cite{agolgm} and on Dani Wise's 
work~\cite{WisePreprint,AFW16,WiseRAAGs}. We further remark one subtlety 
concerning Theorem \ref{thm:geometric} to the reader. Namely, 
a circle bundle over a surface with nonempty boundary can admit 
a geometric structure modeled on both $\bH^2\times\R$ and on 
$\widetilde{\mathrm{PSL}_2(\R)}$ at the same time. Circle bundles 
over surfaces with nonempty boundary are always considered to 
be in the purview of Theorem \ref{thm:geometric}.

Motivated by the topological study of plane algebraic curves 
(see Subsection \ref{subsec:arr curves} below) we isolate a class 
$\mX$ of compact, $3$-dimensional graph manifolds whose 
fundamental groups are RFR$p$.   Namely, a graph manifold 
$M$ lies in the class $\mX$ if the following conditions are satisfied:

\begin{enumerate}[label=($\mX_\arabic*$)]
\item  \label{x1}
The underlying graph $\gam$ is finite, connected, and bipartite with colors $\mP$ and 
$\mL$, and each vertex in $\mP$ has degree at least two.
\item \label{x2}
Each vertex manifold $M_v$ is homeomorphic to a trivial circle bundle over an 
orientable surface with boundary.
\item \label{x3}
If $M_v$ is colored by $\mL$ then at least one boundary component of $M_v$ is a 
boundary component of $M$, and the Euler number of $M_v$ is zero. 
\item \label{x4}
If $M_v$ is colored by $\mP$ then no boundary component of $M_v$ is a boundary 
component of $M$, and the Euler number of $M_v$ is nonzero;
\item \label{x5}
The gluing maps are given by flips.
\end{enumerate}

We refer the reader to Section \ref{sect:3manifold} for more details and 
precise definitions of all technical terms involved in this definition. 
Using Theorem \ref{t:comb}, we prove the following result.

\begin{thm}
\label{thm:graph manifold}
Let $M$ be a compact graph manifold satisfying the above conditions.  
Then for each prime $p$, the group $\pi_1(M)$ is RFR$p$.
\end{thm}

In the above theorem, the assumption that the gluing maps 
of edge spaces be flips is not an assumption by itself. By a recent result of 
Doig and Horn in \cite{DoigHorn}, any gluing map in a graph manifold can 
be made a flip map, at the expense of adding exceptional fibers to the vertex 
spaces. So, the combination of assumptions \ref{x2} and \ref{x5} 
in Theorem \ref{thm:graph manifold} do actually make for a nontrivial 
hypothesis.  Graph manifolds satisfying just these two assumptions 
and some mild condition on the graph $\Gamma$ were shown by 
Schroeder \cite{Schroeder} to possess metrics of non-positive curvature. 

Recently, many authors have studied graph manifolds which are \emph{virtually special}, 
see for instance \cite{PW14,HP15,Liu13}. One of the algebraic consequences of a 
graph manifold being virtually special is that its fundamental group is virtually RFR$p$ 
for each prime $p$, cf.~\cite{AF11,AF13,Ko13}. It is important to note that, although the 
graph manifolds covered by Theorem \ref{thm:graph manifold} are generally virtually 
special, the conclusion of the theorem is not a virtual statement. 
In particular, the theorem does not follow formally from known results, since RFR$p$ 
is a more refined property than virtual specialness. The reader is directed to 
Subsection \ref{subsec:complements} below.

Applying Agol's and Wise's results~\cite{agolrfrs,agolgm,WisePreprint,WiseRAAGs}, 
one obtains the following general fact quite easily:

\begin{cor}
\label{cor:aspherical}
Let $M$ be a compact aspherical $3$-manifold with $\chi(M)=0$. 
Then there exists a finite cover $M'\to M$ such that for each prime $p$, the 
group $\pi_1(M')$ is RFR$p$.
\end{cor}

Note that Corollary \ref{cor:aspherical} is a virtual statement and hence does not imply 
Theorem \ref{thm:graph manifold} formally. To see why Corollary \ref{cor:aspherical} 
holds, it suffices to note that $M$ as in the hypothesis of the corollary is proved by 
Agol to be virtually special, so that some finite index subgroup of $\pi_1(M)$ is 
RFR$p$ for each prime by Proposition \ref{prop:ex rfrp}, part \ref{raag}.

\subsection{Plane algebraic curves}
\label{subsec:arr curves}
The naturality of the manifolds in the purview of Theorem \ref{thm:graph manifold} 
comes from the fact that they are an axiomatized version of boundary manifolds 
of curve arrangements in $\C^2$.  More precisely, let $\mC$ be a (reduced) algebraic 
curve in the complex affine plane. The {\em boundary manifold}\/ of this curve, 
$M_{\mC}$, is obtained by intersecting the boundary of a regular 
neighborhood of $\mC$ with a $4$-ball of sufficiently large radius, 
so that all singularities of $\mC$ are contained in this ball.
Clearly, $M_{\mC}$ is a compact, connected, oriented $3$-manifold. 
Moreover, if each irreducible component of the curve $\mC$ is transverse 
to the line at infinity in $\C^2$, then the boundary components of 
$M_{\mC}$ are tori. 

In Theorem \ref{thm:bdrycurve} we show that, except 
for a few easy-to-handle cases, all boundary manifolds arising 
in the above fashion belongs to the class $\mX$ of graph-manifolds. 
As a consequence, we deduce from Theorem \ref{thm:graph manifold} 
the following result.

\begin{thm}
\label{thm:bdyrfrp}
Let $\mC$ be an algebraic curve in $\C^2$.  Suppose each 
irreducible component of $\mC$ is smooth and 
transverse to the line at infinity, and all singularities of $\mC$ 
are of type ${\rm A}$. Then $\pi_1(M_{\mC})$ is RFR$p$, for all primes $p$.
\end{thm}

Here, by type ${\rm A}$ singularity, we mean that the germs are isomorphic to a pencil of lines.
See Subsection~\ref{subsec:affine}. The following particular case is worth singling out.  

\begin{cor}
\label{cor:bdyarrrfrp}
If $\A$ is an arrangement of lines in $\C^2$, then the 
fundamental group of the boundary manifold of $\A$ 
is RFR$p$, for all primes $p$.
\end{cor}

We also show in Section \ref{sect:boundary manifold}  
that Theorem \ref{thm:bdyrfrp} does not generalize to the compact case, 
namely, that the boundary manifold of an algebraic curve in $\C\bP^2$ 
(even one that satisfies the aforementioned conditions), does not 
always have an RFR$p$ fundamental group.

The motivation behind studying the RFR$p$ property 
for boundary manifolds of arrangements comes from 
the following fundamental unsolved problems:

\begin{prob}
\label{prob:rfrp}
Let $G=\pi_1(\C^2\setminus V(\A))$, where here $\A$ is a line arrangement 
and where $V(\A)$ denotes the variety consisting of the union of the 
lines in $\A$.
\begin{enumerate}
\item
Is $G$ residually finite?
\item
Is $G$ torsion-free?
\end{enumerate}
\end{prob}

The algebraic and topological structure of the complement $\C^2\setminus V(\A)$ 
is closely related to that of the boundary manifold $M_{\A}$ (see~\cite{Hi01, CS08}, 
for instance), and Theorem \ref{thm:bdyrfrp} provides the first step in an 
approach to resolving Problem \ref{prob:rfrp}.

As for a self-contained approach to boundary manifolds of arrangements, 
this paper achieves (among many other things) two goals. The first is that 
it provides an axiomatic setup for studying boundary manifolds. The second 
is a subtlety of virtual versus non-virtual residual properties: the RFR$p$ 
property of fundamental groups of boundary manifolds at every prime $p$ 
holds on the nose, not just after passing to a finite index subgroup.

\subsection{Acknowledgements}
\label{subsec:acknowledgements}
The authors thank K.~Adiprasito, S.~Friedl, E.~Hironaka, 
Y.~Minsky, and A.~Silberstein for helpful conversations. 
The first author was partially supported by NSF grant DMS-1203964 
and Simons Foundation collaborative grant 429836, as well as 
by NSF Grant DMS-1711488 and by an Alfred P. Sloan Foundation Research Fellowship.. 
The second author was partially supported by NSA grant H98230-13-1-0225 
and the Simons Foundation collaborative grant 354156.

\section{Residually finite rationally $p$ groups}
\label{sect:rfrp}

In this section, we give a (very nearly) self-contained account of 
residually finite rationally $p$ groups. 

\subsection{The RFR$p$ filtration}
\label{subsec:rfrp filt}
Let $G=G_1$ be a finitely generated group and let $p$ be a prime.  
We say that $G$ is \emph{residually finite rationally $p$}\/ or RFR$p$ 
if there exists a sequence of subgroups 
$\{G_i\}_{i\geq 1}$ of $G$ such that:
\begin{enumerate}
\item \label{rfr1}
For each $i$, the group $G_{i+1}$ is a normal subgroup of $G_i$.
\item \label{rfr2}
We have 
\[
\bigcap_{i\ge 1} G_i=\{1\}.
\]
\item \label{rfr3}
For each $i$, the group $G_i/G_{i+1}$ is an elementary abelian $p$-group.
\item \label{rfr4}
For each $i$, we have that 
\[
\ker\{G_i\to H_1(G_i,\bQ)\}<G_{i+1}.
\]
\end{enumerate}

The reader may compare the RFR$p$ condition with the RFRS 
condition developed by Agol in \cite{agolrfrs}.  Agol requires each subgroup $G_i$ 
to be normal in $G$ and drops the requirement that $G_i/G_{i+1}$ 
be a $p$-group.

For a general finitely generated, abelian group $K$, let $\T(K)$ 
denote the torsion subgroup of $K$, and let 
\begin{equation}
\label{eq:tfk}
\TF(K)=K/\T(K).
\end{equation}
be the maximal torsion-free quotient of $K$.

\begin{lem}
\label{l:normal}
Let $G$ be RFR$p$ as above with a sequence $\{G_i\}$ of subgroups 
witnessing the statement that $G$ is RFR$p$.  Then there exists a 
sequence of subgroups $\{K_i\}$ of $G$ which witness the fact that 
$G$ is RFR$p$, and such that $K_i$ is normal in $G_1$ for each $i$.
\end{lem}

\begin{proof}
We set $K_1=G$, and we define 
\begin{equation}
\label{eq:ki1}
K_{i+1}=\ker\big\{K_i \to \TF H_1(K_i,\bZ)\to (\TF H_1(K_i,\bZ))\otimes\bZ/p\bZ\big\}.
\end{equation}

By construction, each subgroup $K_i$ is characteristic in $G$, thereby 
verifying condition \eqref{rfr1}.  It is also clear that the sequence 
$\{K_i\}_{i\ge 1}$ satisfies conditions \eqref{rfr3} and \eqref{rfr4}.  
To see that condition  \eqref{rfr2} holds, note that $K_i/K_{i+1}$ 
is the largest elementary abelian quotient of $K_i$ satisfying \eqref{rfr4}.  
It follows immediately that $K_i\leq G_i$, so that 
\begin{equation}
\label{eq:capk}
\bigcap_iK_i\subset\bigcap_i G_i=\{1\},
\end{equation}
whence the conclusion.
\end{proof}

It follows from Lemma \ref{l:normal} that the RFR$p$ condition 
is strictly stronger than Agol's RFRS condition.  
The reader may note that in \cite{agolrfrs}, Agol shows that right-angled Artin 
groups are RFR$p$ for $p=2$, which we will show in Theorem \ref{t:rfrp groups} 
implies that all subgroups of right-angled Artin groups are RFR$2$.

The nested sequence of subgroups $\{K_i\}$ furnished by Lemma \ref{l:normal} 
will be called the \emph{standard RFR$p$ sequence} or the 
\emph{standard RFR$p$ filtration}.  Passing between groups 
and spaces, if $X$ is a connected CW-complex with $\pi_1(X)=G$, and 
if $\{X_i\}$ is a tower of covers such that $\pi_1(X_i)=K_i$, we call 
$\{X_i\}$ the \emph{standard RFR$p$ tower}\/ of $X$.  We will often call the 
quotients $G/K_i$ the \emph{RFR$p$ quotients of $G$}, which is 
not to be confused with those quotients of $G$ which are RFR$p$.

\subsection{The RFR$p$ topology}
\label{subsec:rpfrp top}
Let $p$ be a fixed prime. If $G$ is a finitely generated group, we take the 
natural definition for the \emph{RFR$p$ topology}\/ on $G$. A neighborhood 
basis for the identity is given by the standard RFR$p$ filtration of $G$, and 
a basis for the topology in general is given by the cosets of these subgroups. 
The group $G$ is RFR$p$ if and only if this topology is Hausdorff. We remark
that the use of the terminology ``topology" for us is mostly organizational, and 
we will not require specifics about this topology. As such, we will not delve
deeply into the details of the structure of the topology.

Let $H<G$ be a subgroup, let $\{G_i\}$ be the standard RFR$p$ filtration on $G$, 
and let $\phi_i\colon G\to G/G_i$ be the canonical projection. The subgroup $H$ is \emph{closed}\/ 
in the RFR$p$ topology if and only if for each $g\in G\setminus H$, there is an $i$ 
such that $\phi_i(g)\notin \phi_i(H)$.

If $G$ is a finitely generated group and $H<G$ is a finitely generated subgroup, 
then the \emph{RFR$p$ topology on $H$ induced by $G$ by restriction}\/ is the 
topology on $H$ whose neighborhood basis for the identity is given by the 
subgroups $\{G_i\cap H\}_{i\geq 1}$, where $\{G_i\}_{i\geq 1}$ is the standard 
RFR$p$ filtration of $G$. If $\{H_j\}_{j\geq 1}$ is the standard RFR$p$ filtration 
on $H$, we say that $G$ induces the RFR$p$ topology on $H$ if for each $j$ 
there exists an $i$ such that $H_j>G_i\cap H$.

Let $p$ be a fixed prime, let $G$ be a finitely generated group, and let $\{G_i\}$ 
be the standard RFR$p$ filtration on $G$. We denote the \emph{RFR$p$ radical}\/ 
of $G$ by 
\begin{equation}
\label{eq:radp}
\rad_p(G)=\bigcap_i G_i.
\end{equation}
We have that $G$ is RFR$p$ if and only if $\rad_p(G)$ is trivial. Notice that if 
$i\leq j$ then $\rad_p(G_i)=\rad_p(G_j)$, by the definition of the standard 
RFR$p$ filtration on $G$. The following fact is relatively straightforward, 
but is nevertheless useful:

\begin{prop}
\label{prop:radical}
Let $G$ be a finitely generated group, and let $\phi\colon G\to H$ 
be a surjective homomorphism, where $H$ is RFR$p$. 
If $g\notin\ker\phi$, then $g\notin\rad_p(G)$.
\end{prop}

\begin{proof}
Write $\{G_i\}$ denote the standard RFR$p$ filtration on $G$. Let $1\neq h=\phi(g)$, 
and let $\{H_i\}$ be the standard RFR$p$ filtration of $H$. Then $h\in H_i\setminus H_{i+1}$ 
for some $i$. Pulling back the subgroups $\{H_i\}$ to a collection of subgroups $\{K_i\}$ of $G$, 
we have that $g\in K_i\setminus K_{i+1}$. Moreover, we have that $K_i/K_{i+1}$ is an 
elementary abelian $p$-group, and the quotient map $K_i\to K_i/K_{i+1}$ factors through 
the torsion-free abelianization $K_i\to \TF H_1(K_i,\bZ)$. It follows that for each $i$, 
there exists a $j$ such that $G_j<K_i$, by the same argument as in Lemma \ref{l:normal}. 
It follows that 
\begin{equation}
\label{eq:radpg}
\rad_p(G)<\bigcap_i K_i,
\end{equation}
so that $g\notin\rad_p(G)$.
\end{proof}

The following fact will be useful in the sequel:

\begin{cor}
\label{cor:retract}
Let $G$ be a finitely generated group, and let $r\colon G\to H$ 
be a retraction to a subgroup $H<G$. Then the 
RFR$p$ topology on $G$ induces the RFR$p$ topology on $H$.
\end{cor}

\begin{proof}
Let $\{G_i\}_{i\geq 1}$ be the RFR$p$ filtration on $G$ and let $\{H_i\}_{i\geq 1}$ 
be the RFR$p$ filtration on $H$. Note that $H_1=H\cap G_1$ by definition. 
Assume that $H_i=G_i\cap H$ for some $i$. The retraction $r$ maps $G_i$ onto $H_i$, 
since the inclusion map of $H_i$ into $G_i$ is a right inverse to the identity 
map on $H_i$. 

Notice that $H_i$ maps onto $H_i/H_{i+1}$, and that this 
group is a quotient of $G_i$ which must factor through 
$\TF(H_1(G_i,\Z))\otimes\bZ/p\bZ$.  Hence, $H_{i+1}>H\cap G_{i+1}$. 
Thus the topology on $H$ induced by the filtration $\{G_i\}_{i\geq 1}$ 
is the RFR$p$ topology on $H$.
\end{proof}

\begin{cor}
\label{cor:closed retract}
Let $G$ be a finitely generated RFR$p$ group, and let $\phi\colon G\to H$ 
be a retraction. Then $H$ is closed in the RFR$p$ topology on $G$.
\end{cor}
\begin{proof}
The proof is identical to the proof of Lemma 3.9 in \cite{HW99}. We recall 
a proof for the convenience of the reader. Let $\{G_i\}_{i\geq 1}$ be the 
standard RFR$p$ filtration of $G$, let $N=\ker\phi$, and let $N_i=G_i\cap N$. 
Note that for each $i$, the subgroup $N_i<N$ has finite index. Observe that 
every element of $G$ can be written uniquely as a product $n\cdot h$, where 
$n\in N$ and $h\in H$. It follows that the intersection of the subgroups 
$\{N_iH\}_{i\geq 1}$ is exactly $H$, so that $H$ is closed in the RFR$p$ 
topology on $G$.
\end{proof}

\subsection{Relationship to nilpotent groups}
\label{subsec:nilp}
We note the following fairly straightforward fact. 
The reader may wish to compare the proof of Proposition \ref{prop:rfrp rtfn} 
below with \cite[Lemma 8.3]{Ko12}.

\begin{prop}
\label{prop:rfrp rtfn}
Let $N$ be a non-abelian nilpotent group.  
Then $N$ is not RFR$p$ for any prime $p$.
\end{prop}

\begin{proof}
First, if $N$ has torsion then $N$ is not RFR$p$ for any prime, 
as we will see below in Proposition \ref{prop:rprp props}. 
So, we may assume that $N$ is torsion-free.

Let $\{\gamma_i(N)\}_{i\geq 1}$ denote the lower central series 
of $N$, so that $\gamma_1(N)=N$ and 
\begin{equation}
\label{eq:lcs}
\gamma_{i+1}(N)=[N,\gamma_i(N)].
\end{equation} 
By assumption, this series terminates. 
Also let $\{N_j\}_{j\ge 1}$ be a sequence of subgroups of $N$ which witnesses 
the claim that $N$ is RFR$p$.  Then, in fact, $\{N_j\}_{j\ge 1}$ is the 
standard RFR$p$ filtration for $N$. 

By induction on the length of the lower central series and on 
$i$, it is straightforward to verify that if 
$g\in \gamma_2(N)\setminus\gamma_3(N)$, 
then the image of $g$ is either trivial or torsion in $N_j^{\ab}$, 
so that $g\in N_{j+1}$ for each $j$. This last claim follows 
from the fact that for each $j$, some nonzero power of $g$ 
is a product of commutators in the torsion-free group $N_j$. 
This contradicts the assumption that $\bigcap_j N_j=\{1\}$.  
\end{proof}

Furthermore, whether or not a particular group enjoys the RFR$p$ property 
depends on the prime $p$:

\begin{prop}
\label{p:gp}
For each prime $p$, there exists a finitely presented group $G_p$ which is RFR$p$, 
but $G_p$ is not RFR$q$ for any prime $q\neq p$.
\end{prop}

\begin{proof}
Fix a basis $\{v_1,\ldots,v_p\}$ for $\bZ^p$.  Let 
\begin{equation}
\label{eq:regrep}
\bZ/p\bZ\to \GL_p(\bZ)=\Aut(\bZ^p)
\end{equation}
be the regular representation of $\bZ/p\bZ$ which permutes the coordinates of 
$\bZ^p$.  We consider the $\bQ$-irreducible representation $V\cong \bZ^{p-1}\otimes\bQ$ 
of $\bZ/p\bZ$ given by vectors whose coordinates add up to zero, and let 
$A\subset V$ be the integral points.  Furthermore, we let $G_p$ be the 
semidirect product 
\begin{equation}
\label{eq:agz}
\xymatrixcolsep{14pt}
\xymatrix{1\ar[r]& A \ar[r]& G_p \ar[r]&\bZ \ar[r]& 1}, 
\end{equation}
where the $\bZ$-action on $A$ is via the canonical projection $\bZ\surj\bZ/p\bZ$.

Observe that $b_1(G_p)=1$.  Observe furthermore that the kernel of 
the map $G_p\to\bZ/p\bZ$ given by reducing the first homology of $G_p$ 
modulo $p$ is isomorphic to $\bZ^p$.  It is clear then that $G_p$ is RFR$p$.

Now let $q\neq p$ be another prime.  Let $G_p\surj \bZ/q\bZ$ be the map given 
by reducing the homology of $G_p$ modulo $q$, and let $K_{q}$ be the kernel.  
Since $q$ and $p$ are relatively prime, and since $V$ is an irreducible 
$\bQ$-representation of $\bZ/p\bZ$, we have that $K_{q}\cong G_p$.  It follows that 
$G_p$ is not RFR${q}$, since $A$ is contained in any sequence of subgroups 
witnessing the claim that $G_p$ is RFR$q$.
\end{proof}

Observe that in Proposition \ref{p:gp}, the group $G_p$ is virtually abelian.  
To obtain a non-virtually abelian example, 
observe that for each $n$, the natural homomorphism $\Aut(F_n)\to \GL_n(\bZ)$ 
is surjective.  Thus, one can mimic the construction of Proposition \ref{p:gp} in 
the free group case, obtaining a semidirect product 
\begin{equation}
\label{eq:fhz}
\xymatrixcolsep{14pt}
\xymatrix{1\ar[r]& F_{p-1} \ar[r]&  H_p \ar[r]& \bZ \ar[r]&  1}
\end{equation}
of $\bZ$ with a free group of rank $(p-1)$ which is RFR$p$ for exactly one prime.  
It is easy to check that $H_p$ is virtually a direct product, and that neither $G_p$ 
nor $H_p$ is residually torsion-free nilpotent.

\section{Properties of RFR$p$ groups}
\label{sec:tfq rfrp}

In this section we discuss further properties enjoyed by residually 
finite rationally $p$ groups. 

\subsection{Torsion-free quotients of RFR$p$ groups}
\label{subsec:tfq rfrp}

We start with some immediate consequences of the definition.

\begin{prop}
\label{prop:rprp props}
Let $G$ be a finitely generated group which is RFR$p$.  Then:
\begin{enumerate}
\item  \label{r1}
$G$ is residually $p$.  In particular, $G$ is residually finite and 
residually nilpotent.
\item  \label{r2}
$G$ is torsion-free.
\item \label{r3}
If in addition $G$ is finitely presented, then $G$ has a solvable word problem.
\end{enumerate}
\end{prop}

\begin{proof}
Item \eqref{r1} is straightforward from the definition.  Item \eqref{r2} 
follows from the fact that if $1\neq g\in G_i\setminus G_{i+1}$, then 
$g$ represents a torsion-free class in $H_1(G_i,\Z)$, and therefore 
has infinite order in $G$.  Item \eqref{r3} is a completely standard result about 
finitely presented, residually finite groups (see for instance~\cite{LyndonSchupp}).
\end{proof}

If a finitely generated group $G$ is RFR$p$ for every prime $p$, we have that 
$G$ is residually $p$ for every prime and torsion-free.  Recall that a group 
$G$ is \emph{residually torsion-free nilpotent}\/ if each non-identity element 
$g\in G$ survives in a torsion-free nilpotent quotient of $G$.  Note that a group 
which is residually torsion-free nilpotent is torsion-free and residually $p$ for 
every prime $p$.  

Residual torsion-free nilpotence of a finitely generated group 
$G$ is a rather strong property which has many useful consequences.  
For instance, if $G$ is residually torsion-free nilpotent then $G$ is bi-orderable 
and $\bZ[G]$ is an integral domain (see~\cite{DeroinNavasRivas,Dehornoyetal}). 

\begin{que}
\label{quest:rtfn}
Let $G$ be a finitely generated group which is RFR$p$ for every prime $p$.  
Is $G$ residually torsion-free nilpotent?
\end{que}

Recall that a group $G$ is \emph{polycyclic}\/ if it admits a finite subnormal 
series with cyclic factors.  Note that a finitely generated nilpotent group is 
polycyclic and that a polycyclic group is solvable, but that the reverse 
implications are generally false.

\begin{thm}
\label{thm:rtfpoly}
Let $G$ be a finitely generated group which is RFR$p$ for some prime $p$.  
Then $G$ is residually torsion-free polycyclic.  In particular, $G$ is residually 
torsion-free solvable.
\end{thm}

\begin{proof}
Let $\{G_i\}$ be the standard RFR$p$ sequence for $G$, with $G_1=G$.  
We will define a sequence $\{K_i\}$ of subgroups of $G$ such that $K_i\leq G_i$ 
for each $i$ and such that $G/K_i$ is a torsion-free polycyclic group for each $i$.  
Since $\bigcap_i G_i=\{1\}$, this will prove that $G$ is residually torsion-free polycyclic.

Set $K_1=G_1$ and set 
\[
K_2=\ker\{G_1\to\TF H_1(G_1,\bZ)\}<G_2.
\]  
Since $G$ is finitely generated, we have that $G/K_2$ is a finitely generated 
torsion-free abelian group and therefore torsion-free polycyclic.  In general, 
we set 
\[
K_{i+1}=(\ker\{G_i\to\TF H_1(G_i,\bZ)\})\cap K_i<G_{i+1}.
\]  
By the Second Isomorphism Theorem for groups, we have that 
\[
K_i/K_{i+1}\cong \frac{K_i\cdot (\ker\{G_i\to\TF H_1(G_i,\bZ)\})}{\ker\{G_i\to
\TF H_1(G_i,\bZ)\}}.
\]  

Since $K_i<G_i$, we have that $K_i/K_{i+1}$ is a subgroup of the finitely 
generated abelian group $\TF H_1(G_i,\bZ)$ and is therefore a finitely 
generated, torsion-free abelian group.  By construction, $K_i$ is normal 
in $G$ for each $i$, so that by induction on $i$, we have that $G/K_i$ 
is torsion-free polycyclic for each $i$.
\end{proof}

We note the residual torsion-free solvability of RFR$p$ groups because 
of apparent connections to BNS invariants~\cite{friedltillman}.

\subsection{Separability of maximal abelian subgroups}
\label{subsec:max abel}
In the proof of Theorem \ref{thm:combination}, we will require 
the separability of certain subgroups of RFR$p$ groups.  
Let $G$ be a group.  We will say that a subgroup $H<G$ is 
\emph{separable}\/ if for every $g\in G\setminus H$, there 
is a finite quotient $\phi\colon G\to Q$ such that $\phi(g)\notin\phi(H)$.  
A subgroup $H$ is \emph{RFR$p$-separable} in $G$ if 
we can assume that $Q=G/G_i$ for some term $G_i$ in 
the standard RFR$p$ filtration of $G$. In other words, 
a subgroup $H<G$ is RFR$p$-separable in $G$ if and 
only if $H$ is closed in the RFR$p$ topology on $G$.

The following result about RFR$p$ groups mirrors a result of 
E.~Hamilton about hyperbolic $3$-manifold groups (see \cite{hamilton}):

\begin{thm}
\label{thm:separable}
Let $G$ be a finitely generated RFR$p$ group and let $K<G$ be a finitely 
generated abelian subgroup which is maximal among abelian groups with 
the same rank as $K$. Then $K$ is RFR$p$-separable in $G$.
\end{thm}

The maximality assumption on $K$ simply means that if $K$ is properly contained 
in an abelian subgroup $H<G$ then $\rk K<\rk H$. The necessity of this assumption 
results from the following example: suppose $K$ and $H$ are both torsion-free 
abelian groups of rank $n$, and that $p$ does not divide $[H:K]$. 
Then $K$ is not separable in the RFR$p$ topology on $H$, even though the 
RFR$p$ topology on $K$ agrees with the RFR$p$ topology induced from $H$.

\begin{proof}[Proof of Theorem \ref{thm:separable}]
As usual, write $\{G_i\}_{i\geq 1}$ for the standard RFR$p$ filtration of $G$, 
and write $\phi_i\colon G\to G/G_{i+1}$ for the canonical projection. Let 
$g\in G\setminus K$.   Since $K$ is 
maximal with respect to abelian subgroups of $G$ of the same rank as $K$, 
the group $H:=\langle g,K\rangle$ 
is not isomorphic to $K$. We write $n$ for the rank of $K$, so that either 
$H$ is abelian of rank $n+1$, or $g$ does not centralize $K$.

Notice that if $g$ is not in the centralizer of $K$ then there is some $k\in K$ 
such that $[g,k]\neq 1$. Then for some $i$, we have $[g,k]\in G_i\setminus G_{i+1}$. 
In particular, $\phi_i([g,k])\neq 1$ in $G/G_i$, so that $\phi_i(g)\notin \phi_i(K)$. 
Thus, if $g$ does not centralize $K$ then we can separate $g$ from $K$ in the 
RFR$p$ topology.

Thus, we may assume that $g$ centralizes $K$, so that $H\cong \bZ^{n+1}$. 
It suffices to find an $i$ such that $\phi_i(H)$ is an abelian $p$-group of rank 
exactly $n+1$. This way,  since $\phi_i(K)$ will be a $p$-group of rank at most $n$, 
we will immediately obtain that $\phi_i(g)\notin\phi_i(K)$, thereby showing that  $g$ 
may be separated from $K$ in the RFR$p$ topology. 

We proceed by induction on $n$.   The base case of the 
induction is clear. If $H$ is a cyclic group, then because $G$ is 
RFR$p$ and $g\in G\setminus\{1\}$, there is an $i$ for which 
$\phi_i(H)$ will be a $p$-group of rank exactly one.

For the inductive step, we may assume that there is an $i$ such that $\phi_i(K)$ is a 
$p$-group of rank exactly $n$. Picking a basis $\{k_1,\ldots, k_n,g\}$ for $H$, we 
may choose an index $j\geq i$ such that no basis element for $H$ lies in in $G_{j+1}$. 
The basis elements themselves may not lie in $G_j$, but by replacing the chosen basis 
elements by positive powers if necessary, we may assume they do. We then consider
 the image $\bar{H}$ of $H$ inside of $\TF(G_j^{\ab})$.

Observe that since $j\geq i$, we have that $\rk\bar{H}$ is at least $n$, by the assumptions 
on $i$. Note furthermore that if $\rk\bar{H}=n+1$, then the image of $\bar{H}$ in 
\begin{equation}
\label{eq:tfgj}
\TF(G_j^{\ab})\otimes \bZ/p^m\bZ
\end{equation}
will have rank $n+1$ for some sufficiently large $m$. By the definition of the RFR$p$ filtration 
on $G$, we have that the image of $\phi_k(H)$ in $G/G_{k+1}$ will contain an abelian 
$p$-group of rank exactly $n+1$ for some $k\geq j$, and will thus itself have rank exactly $n+1$.

Thus, we may assume that $\rk\bar{H}=n$, so that there is an element $d\in H$ which is 
nontrivial in $H$ but which is trivial in $\bar{H}$. We choose an index $s$ such that 
$d\notin G_{s+1}$ and again replace the basis elements of $H$ and $d$ by suitable 
powers so that they lie in $G_s$. Then, (suitable powers of) the elements $k_1,\ldots,k_n,h,d$ 
generate a subgroup of $\TF (G_s^{\ab})$ of rank at least $n$, and $d$ lies in the kernel 
of the natural map  
\begin{equation}
\label{eq:tags}
\xymatrixcolsep{16pt}
\xymatrix{
\TF (G_s^{\ab})\ar[r]&\TF (G_j^{\ab})}
\end{equation}
induced by the inclusion $G_s\to G_j$. The image of $H$ in $\TF (G_j^{\ab})$ under 
this map has rank exactly $n$. Thus, the image of $H$ in $\TF (G_s^{\ab})$ must 
have rank exactly $n+1$.

Again, we see that the image of $H$ in $\TF(G_s^{\ab})\otimes \bZ/p^m\bZ$ has rank 
$n+1$, for some sufficiently large $m$. In particular, $\phi_t(H)$ is an abelian 
$p$-group of rank $n+1$ for some $t\geq s$, and this completes the proof.
\end{proof}

\subsection{Orderability}
\label{subsec:orderability}

In this subsection, we establish item \eqref{res9} of Theorem \ref{t:rfrpsummary}. 
This follows from the following result of Rhemtulla~\cite{Rhemtulla1973}:

\begin{thm}[\cite{Rhemtulla1973}]
\label{thm:orderability}
Let $G$ be a group which is residually locally residually $p$ for 
infinitely many primes $p$. Then $G$ is bi-orderable.
\end{thm}

The consequence of this result which is relevant for our discussion is the following:

\begin{cor}
\label{cor:biorder}
Let $G$ be a group which is RFR$p$ for infinitely many primes $p$. 
Then $G$ is bi-orderable.
\end{cor}

\section{Classes of groups which are RFR$p$}
\label{sec:class rfrp}

We now populate the class of RFR$p$ groups with several families 
of examples occurring in low-dimensional topology and geometric 
group theory. 

\subsection{Free groups and surface groups}
\label{subsec:free surf}

We start by showing that finitely generated free groups 
are residually finite rationally $p$, based on an argument  
the first author gave in \cite{IMSProc}, which we will recall 
for the convenience of the reader.   

\begin{prop}
\label{prop:free rfrp}
Finitely generated free groups are RFR$p$, for all primes $p$. 
\end{prop}

\begin{proof}
We realize a free group as the fundamental 
group of a wedge of circles $X$, which we think of as a graph equipped with 
the graph metric.  In any simplicial graph $\gam$ equipped with the graph 
metric, we have the following two observations: first, any shortest unbased 
(non-backtracking) loop $\gamma$ in $\gam$ is simple, i.e., $\gamma$ has 
no self-intersections.  Second, if $\gamma$ is a simple, oriented loop, then 
the homology class $[\gamma]\in H_1(\gam,\bZ)$ is primitive.  This can be 
seen by choosing any edge $e$ of $\gamma$, extending $\gamma\setminus e$ 
to a maximal tree $T\subset\gam$, and considering the graph $\gam/T$.  

Using these observations, we build a sequence of covers of $X$ 
by setting $X_1=X$ and letting $X_{i+1}$ be the finite cover of 
$X_i$ induced by the quotient $H_1(X_i,\bZ/p\bZ)$ 
of $\pi_1(X_i)$.  We see that the shortest unbased loop in $X_i$ does not 
lift to $X_{i+1}$, so that by induction, the shortest loop in $X_i$ has length 
at least $i$ in the graph metric.
\end{proof}

More generally, we will show in Proposition \ref{prop:raag rfrp}
that right-angled Artin groups are residually finite rationally 
$p$, for all primes $p$. 

\begin{prop}
\label{prop:free surf}
Fundamental groups of closed, orientable surfaces 
are RFR$p$, for all primes $p$.
\end{prop}

\begin{proof}
The argument is nearly identical to that for free groups.  
For genus one, the claim is straightforward, so we assume the genus of the base 
surface to be at least two.  We choose a hyperbolic metric on a surface $X=X_1$ 
and on all of its covers (by pullback).  Again, any shortest geodesic on a hyperbolic 
surface is simple.  If $\gamma$ is a simple, oriented, closed geodesic on a hyperbolic 
surface, $\gamma$ represents a primitive homology class if and only if it is non-separating.  
If $\gamma\subset X$ is a separating simple closed geodesic and $p$ is any prime, then 
$\gamma$ lifts to the universal modulo $p$ homology cover $X_p\to X$, and any lift of 
$\gamma$ is non-separating on $X_p$ (though the union of all lifts of $\gamma$ 
is separating).  

We now build the tower of covers $\{X_i\}$ of $X=X_1$ in the same 
manner as in the case of free groups.  The hyperbolic length spectrum 
of geodesics on $X$ is a discrete subset of $\bR$, since a hyperbolic 
metric is induced by a discrete representation of $\pi_1(S)\to \textrm{PSL}_2(\bR)$.    
Thus, we again see that for any closed geodesic $\gamma\subset X$, we have that 
$\gamma$ does not lift to $X_i$ for $i\gg 1$.
\end{proof}

\subsection{Operations on RFR$p$ groups}
\label{subsec:ops}

Next, we show that the class of RFR$p$ groups is closed under 
certain natural operations. A nearly verbatim statement for RFRS 
groups was established by Agol in \cite{agolrfrs}, albeit our proof for part 
\eqref{coprod} is somewhat different.

\begin{thm}
\label{t:rfrp groups}
Fix a prime $p$.  The class of RFR$p$ groups is closed 
under the following operations:
\begin{enumerate}
\item \label{subgroup}
Passing to finitely generated subgroups.
\item \label{prod}
Taking finite direct products.
\item \label{coprod}
Taking finite free products.
\end{enumerate}
\end{thm}

We remark that an arbitrary subgroup of an RFR$p$ group will be RFR$p$ 
in an appropriate sense; only the finite generation may be lost.

\begin{proof}[Proof of Theorem \ref{t:rfrp groups}]
We prove the items in order.  Let $G$ be a group which is RFR$p$, as witnessed 
by a sequence of nested subgroups $\{G_i\}$, and let $H<G$ be an arbitrary subgroup.  
We set $H_i=H\cap G_i$.  Evidently, we have 
\[
\bigcap_i H_i=\{1\}.
\]  
Furthermore, $H_i/H_{i+1}$ is the image of $H_i$ inside of $G_i/G_{i+1}$ and is 
therefore an elementary abelian $p$-group. Moreover, if $h\in H_i\setminus H_{i+1}$,  
then $h\in G_i\setminus G_{i+1}$ and therefore has infinite order in $H_1(G_i,\bZ)$.  
It follows that $h$ must also have infinite order in $H_1(H_i,\bZ)$.  Thus, the sequence 
of subgroups $\{H_i\}$ witnesses the claim that $H$ is RFR$p$.  Thus, the class of 
RFR$p$ groups  is closed under taking subgroups.

If $G$ and $H$ are groups, we have 
\[
H_1(G\times H,\bZ)\cong H_1(G,\bZ)\times H_1(H,\bZ).
\]  
If $G$ and $H$ are RFR$p$ with nested sequences of subgroups $\{G_i\}$ and $\{H_i\}$, 
we set $K=G\times H$ and $K_i=G_i\times H_i$ for each $i$.  We have that 
\[
K_i/K_{i+1}\cong G_i/G_{i+1}\times H_i/H_{i+1},
\] 
as follows from an easy computation, so that $K_i/K_{i+1}$ 
is an elementary abelian $p$-group.  Furthermore, $K_i/K_{i+1}$ is a quotient of 
$\TF H_1(G_i\times H_i,\bZ)$, and 
\[
\bigcap_i K_i=\bigcap_i G_i\times H_i=\{1\},
\] 
so that $\{K_i\}$ witnesses the fact that $K$ is RFR$p$.  By an easy induction, 
this shows that the class of 
RFR$p$ groups  is closed under taking finite direct products.

To prove that the class the class of 
RFR$p$ groups  is closed under taking finite free products, 
we will note that this claim is a special case of Theorem \ref{thm:combination}. 
It will not be circular to postpone the proof until then (see Corollary \ref{cor:free product}).
\end{proof}

\subsection{Circle bundles over surfaces}
\label{subsec:circle bundles}

The following fact shows that there is a sharp dichotomy between groups which 
are RFR$p$ and groups which are not RFR$p$ in the class of cyclic central 
extensions of surface groups:

\begin{thm}
\label{thm:circle bundle}
Let $S$ be an aspherical, compact, orientable surface and let 
\[
\xymatrixcolsep{14pt}
\xymatrix{
F\ar[r]& E \ar[r]& S}
\] 
be a fiber bundle with fiber $F=S^1$, such that the total space $E$ is orientable. 
Write $e\in H^2(S,\bZ)$ for the Euler class of the bundle.
\begin{enumerate}
\item
If $e=0$, then $\pi_1(E)$ is RFR$p$ for every prime $p$.
\item
If $e\neq 0$ then $\pi_1(E)$ is not RFR$p$ for any prime $p$.
\end{enumerate}
\end{thm}
\begin{proof}
If $e=0$ then $\pi_1(E)\cong \Z\times\pi_1(S)$.  Thus, $\pi_1(E)$ is RFR$p$ for 
every prime $p$ by combining Proposition \ref{prop:ex rfrp} and 
Theorem \ref{t:rfrp groups}.

If $e\neq 0$ then we follow the argument given in \cite{Ko12}, which 
we reproduce here for the reader's convenience. We have a short exact 
sequence
\begin{equation}
\label{eq:zes}
\xymatrixcolsep{14pt}
\xymatrix{1\ar[r] &\Z \ar[r] & \pi_1(E)\ar[r] & \pi_1(S)\ar[r] & 1},
\end{equation} 
where the leftmost copy of $\Z$ is central and is generated by an element $t$. 
We claim that for any prime $p$, we have $\rad_p(\pi_1(E))=\langle t\rangle$.

First, if $g\in\pi_1(E)$ then we may write $g=h\cdot t^k$, where $h\in\pi_1(S)$. 
Since $\pi_1(S)$ is RFR$p$ for each prime $p$, we have that if $h\neq 1$ then 
$h\notin\rad_p(\pi_1(E))$, by Proposition \ref{prop:radical}. Thus, we have an 
inclusion $\langle t\rangle\supset\rad_p(\pi_1(E))$, 
which holds for every prime $p$.

Conversely, write $G=\pi_1(E)$ and $\{G_i\}_{i\geq1}$ for the standard RFR$p$ 
filtration of $G$. The fact that $e\neq 0$ means that a nonzero power of $t$ is a 
product of commutators in $G$ (see~\cite{Br}). In particular, we have that $t$ maps to a torsion 
element of $H_1(G,\Z)$, so that $t\in G_2$. By induction, we may suppose that 
$t\in G_i$ for some $i>1$. Since $G_i<G_1=G$ has finite index, we have that 
$G_i$ again decomposes as a nonsplit central extension of a surface group, 
by a standard cohomology of groups argument using the fact that $H^2(\pi_1(S),\Z)$ 
is torsion-free. In particular, $t$ maps to a torsion element of $H_1(G_i,\Z)$, 
so that $t\in G_{i+1}$. Hence $t\in\rad_p(\pi_1(E))$, establishing the reverse inclusion.
\end{proof}

\subsection{Complements}
\label{subsec:complements}
In this short subsection, we gather some classes of groups which fail to be RFR$p$, 
be it for some prime, for infinitely many primes, or for all primes. Many of the details 
have been discussed above (see Subsections \ref{subsec:rfrp filt} and \ref{subsec:nilp}, 
as well as Theorem \ref{thm:circle bundle}), so we can safely omit most of the proofs.

\begin{prop}\label{prop:one prime}
The following classes of finitely generated groups fail to be RFR$p$ for any prime $p$:
\begin{enumerate}
\item
Groups with torsion.
\item
Nonabelian nilpotent groups.
\item
Central extensions which are not virtually split.
\item
Nonabelian groups $G$ with $b_1(G)<2$.
\end{enumerate}
\end{prop}
\begin{proof}
Only the last item has not been formally established above. If $G$ is nonabelian 
and RFR$p$ then $G$ surjects onto a nonabelian $p$-group, whose abelianization 
cannot be cyclic by elementary group cohomology considerations. Thus, $b_1(G)\geq 2$.
\end{proof}

All the statements in the following proposition have been (or will be) established 
(see Subsection \ref{subsec:orderability} above and Subsection \ref{subsec:tits alt} below).

\begin{prop}
\label{prop:many primes}
The following classes of finitely generated groups fail to be RFR$p$ for 
infinitely many primes:
\begin{enumerate}
\item
Groups which are not bi-orderable.
\item
Groups which are not large.
\end{enumerate}
\end{prop}

Since bi-orderable groups are somewhat rare in nature, the first item of 
Proposition \ref{prop:many primes} really does suggest that RFR$p$ is 
a fine property for groups to enjoy. For instance, many virtually special 
groups arising in $3$-manifold topology fail to be bi-orderable. Thus, 
whereas a graph manifold may be virtually special, knowing its fundamental 
group is RFR$p$ for all primes is a significantly different bit of data, 
which gives further strength to the results above, such as 
Theorem \ref{thm:graph manifold}.

\section{Alexander varieties, BNS invariant, and largeness}
\label{sec:cvs}

We start this section by reviewing some background on the 
homology jump loci and the Alexander varieties of 
spaces and groups, following \cite{Hi97, PS10, Su-imrn}.

\subsection{Jump loci and Alexander invariants}
\label{subsec:jump}

Let $G$ be a finitely-generated group, and let $X$ be a connected  
CW-complex with finite $1$-skeleton such that $\pi_1(X)\cong G$.  
The \emph{characteristic varieties}\/ $V_i(X)$ are the jumping loci for 
the (degree $1$) cohomology groups of $X$ with coefficients in 
rank $1$ local systems.  We write $\yh{G}$ for the group of complex 
characters of $G$.  Let $\chi\colon G\to \bC^*$ be a character of $G$ 
and let $H^1(X,\bC_{\chi})$ be the twisted cohomology 
module corresponding to $\chi$.  For each $i\ge 0$, put 
\begin{equation}
\label{eq:vix}
V_i(X)=\{\chi \in  \yh{G}\mid\dim H^1(X,\bC_{\chi})\geq i\}.
\end{equation}

It is readily seen that each of these sets is a Zariski closed subset of 
the character group; moreover, $V_i(X)\supseteq V_{i+1}(X)$ for all $i$, 
and $V_i(X)=\emptyset$ for $i\gg 0$.  
Clearly, each of these sets depends only on the fundamental 
group of $X$, so we may define $V_i(G):=V_i(X)$. 
If $\phi\colon G_1 \to G_2$ is a surjective homomorphism, 
we obtain an injective morphism $\hat\phi\colon \yh{G}_2 \to \yh{G}_1$ 
by precomposition. It is readily verified that the map $\hat\phi$ takes $V_i(G_2)$ to 
$V_i(G_1)$.

By definition, the trivial representation $\yh{1}\in \yh{G}$ belongs to $V_i(G)$ 
if an only if $b_1(G)\ge i$.  Away from $\yh{1}$, the sets $V_i(G)$ 
coincide with the {\em Alexander varieties}\/ of $G$.  

To define these varieties, 
first consider the derived series of $G$, defined inductively by setting $G'=[G,G]$, 
$G^2=G''=[G',G']$, and $G^k=[G^{k-1},G^{k-1}]$ for $k\ge 3$.  The quotient 
$G/G^k$ is the \emph{universal $k$-step solvable quotient}\/ of $G$.  

Next, let $B(G)=G'/G''$ be the {\em Alexander invariant}\/ of $G$, viewed as a 
module over $\Z[G/G']$ via the conjugation action of $G/G'$ on 
$G'/G''$.   Note that $\C[G/G']$ is the coordinate ring of the character 
group $\yh{G}$.  We then let the $i$-th Alexander variety of $G$ be the 
support locus of the $i$-th exterior power of the complexified Alexander 
invariant, that is, 
\begin{equation}
\label{eq:supp}
W_i(G)=V\bigg({\rm ann} \Big( \bigwedge ^i B(G) \otimes \C \Big)\bigg).  
\end{equation}
As shown in \cite{Hi97}, the following equality holds, for each $i\ge 1$:
\begin{equation}
\label{eq:vw}
V_i(G) \setminus \yh{1} = W_i(G)  \setminus \yh{1}. 
\end{equation}

This description makes it apparent that the characteristic varieties 
$V_i(G)$ only depend on $G/G''$, the maximal metabelian quotient 
of $G$.  More precisely, we have the following lemma.

\begin{lemma}
\label{lem:metacv}
For any finitely generated group $G$, the projection map $\pi\colon G\to G/G''$ 
induces an isomorphism $\hat\pi\colon \yh{G/G''} \to \yh{G}$ which restricts to 
isomorphisms  $V_i(G/G'')\to V_i(G)$ for all $i\ge 1$.
\end{lemma}

\begin{proof}
Clearly, the map $\pi$ induces an isomorphism on abelianizations, and thus 
an isomorphism between the respective character group. 

Now note that $(G/G'')'=G'/G''$ and $(G/G'')''$ is trivial; thus, the map $\pi$ 
also induces an isomorphism $B(G)\to B(G/G'')$.  Applying formulas 
\eqref{eq:supp} and \eqref{eq:vw} proves the remaining claim. 
\end{proof}

\subsection{Non-finitely presented metabelian quotients}
\label{subsec:afp metabelian}

Once again, let $X$ be a connected CW-complex with finite $1$-skeleton, 
and set $\pi_1(X)=G$.  
If $A$ is a finite abelian group and $\phi\colon G\to A$ is a surjective 
homomorphism, we obtain a finite cover $X_A\to X$ induced by $\phi$.  
The complex Betti number $b_1(X_A)$ is related to $b_1(X)$ and the varieties 
$V_i(X)$ by the following well-known formula: 
\begin{equation}
\label{eq:b1cover}
b_1(X_A)=b_1(X)+\sum_{i=1}^k \abs{\yh{\phi}(\yh{A}\setminus \yh{1})\cap V_i(X)},
\end{equation}
where here $V_i(X)=\emptyset$ for $i>k$.

The following result relates torsion points on the Alexander variety to largeness 
for finitely presented groups:

\begin{thm}[See \cite{Ko14}]
\label{t:KoThesis}
Let $G$ be a finitely presented group.  The group $G$ is large if and only if 
there exists a finite index subgroup $H<G$ such that $V_1(H)$ has infinitely 
many torsion points.
\end{thm}

The finite presentation assumption in the above theorem is essential.  
For instance, let $F_n$ be a free group of rank $n\ge 2$.  It is readily 
verified that $V_1(F_n)=(\C^*)^n$.  Thus, by Lemma \ref{lem:metacv}, we also have 
that $V_1(F_n/F_n'')=(\C^*)^n$. In particular, the variety $V_1(F_n/F_n'')$ 
has infinitely many torsion points, though the group $F_n/F_n''$ is 
solvable, and thus not large.

As an application, we can prove the following corollary, which recovers 
and generalizes the main result of Baumslag and Strebel \cite{BS76}.

\begin{cor}
\label{c:rfrpfg}
Let $G$ be a finitely generated group which is nonabelian and RFR$p$ 
for infinitely many primes $p$.  Then the universal metabelian quotient 
$G/G''$ is not finitely presented.  In particular, $G'$ is not finitely generated.
\end{cor}

\begin{proof}
Since $G$ is RFR$p$ for infinitely many primes and not 
abelian, we have that $V_1(G)$ contains infinitely many torsion points, as will follow from Lemma~\ref{lem:tpoints} below.  
By Lemma \ref{lem:metacv}, we have that $V_1(G/G'')$ also contains 
infinitely many torsion points.

Now suppose $G/G''$ is finitely presented.  Then Theorem \ref{t:KoThesis} 
implies that $G/G''$ is large.  However, $G/G''$ is solvable, and this is 
a contradiction.
\end{proof}

A related result (based on the above proof) is given in  
Proposition 4.9 from \cite{PS18}. In fact, the same proof as in 
Corollary \ref{c:rfrpfg} works for all universal solvable quotients:

\begin{cor}
\label{c:rfrpfg-k}
Let $G$ be a finitely generated group which is nonabelian 
and RFR$p$ for infinitely many primes $p$.  Then the universal 
$k$-step solvable quotient $G/G^k$ is not finitely presented, 
for any $k\ge 2$. 
\end{cor}

\subsection{A Tits Alternative for RFR$p$ groups}
\label{subsec:tits alt}
We now connect the RFR$p$ property of a group $G$ to the aforementioned 
arithmetic property of $V_1(G)$. 
 
\begin{lemma}
\label{lem:tpoints}
Let $G$ be a non-abelian, finitely generated group which is RFR$p$ for infinitely 
many primes.  Then $V_1(G)$ contains infinitely many torsion points.  
\end{lemma}

\begin{proof}
Suppose $G$ is RFR$p$ for infinitely many primes $p$.  
For each prime $p$, we write 
\begin{equation}
\label{eq:kpn}
K_{p,n+1}=\ker\{G\to (\TF H_1(G,\bZ))\otimes\bZ/p^n\bZ\}.
\end{equation}

We claim that if $G$ is nonabelian and RFR$p$, then $b_1(K_{p,n})>b_1(G)$ 
for $n\gg 0$.  Indeed, otherwise we can construct a sequence of subgroups 
$\{G_i\}_{i\geq 0}$ which witness the fact that $G$ is RFR$p$, so that 
$G_1=G$ and $G_2=K_{p,2}$.  Since $b_1(K_{p,n})=b_1(G)$ for all $n$, 
an easy induction shows that $G_n=K_{p,n}$ for all $n$.  In particular, 
$\bigcap_n K_{p,n}=\{1\}$, 
which implies $G$ is abelian, since $G'<K_{p,n}$ for all $n$. This is a contradiction.

Thus, if $G$ satisfies out hypothesis, we have that $V_1(G)$ contains 
at least one $p$-torsion point for infinitely many values of $p$, 
by \eqref{eq:b1cover}. Since for primes $p\neq p'$, the $p$-torsion 
and $p'$-torsion points on $V_1(G)$ are disjoint, we have that 
$V_1(G)$ contains infinitely many torsion points.
\end{proof}

Agol asked the first author \cite{agolprivate} whether a group 
which is RFRS and not virtually abelian is \emph{large}, i.e., virtually 
surjects to a nonabelian free group.   We give the following affirmative 
partial answer to Agol's question, which contrasts sharply with the 
example described in Proposition \ref{p:gp}:

\begin{thm}
\label{t:rfrplarge}
Let $G$ be a finitely presented group which is RFR$p$ for infinitely 
many primes.  Then either:
\begin{enumerate}
\item \label{tits1}
$G$ is abelian.
\item \label{tits2}
$G$ is large.
\end{enumerate}
\end{thm}

\begin{proof}
Follows at once from Theorem \ref{t:KoThesis} and Lemma \ref{lem:tpoints}.
\end{proof}

We remark that in addition to largeness of groups, one can consider the class 
of \emph{very large}\/ groups, which are ones which surject to a nonabelian 
free group without passing to a finite index subgroup. All examples
of groups known to the authors which are finitely presented and RFR$p$ 
for infinitely many primes are either abelian or very large. We do not know 
if there is an example of such a group which is large but not very large.

\subsection{$\Sigma$-invariants}
\label{subsec:bns}
We now relate the RFR$p$ property to the Bieri--Neumann--Strebel invariant 
of \cite{bns}.  Once again, let $G$ be a finitely generated group.  Without 
loss of generality, we may assume that 
$G$ is generated by a finite, symmetric set $\Omega$.  We write 
$\cay(G,\Omega)$ for the Cayley graph of $G$ with respect to $\Omega$, 
and we write $S(G)$ for the unit sphere in the first real cohomology 
group of $G$: 
\begin{equation}
\label{eq:sg}
S(G)=(H^1(G,\bR)\setminus\{0\})/\{\chi\sim \lambda\cdot\chi,\,\lambda\in\bR_{>0}\}.
\end{equation} 

For $\chi\in S(G)$, we write $\cay_{\chi}(G,\Omega)$ for the subgraph consisting of 
vertices $g\in G$ such that $\chi(g)\geq 0$.  A fundamental fact about 
this graph is that its connectivity is independent of the generating 
set $\Omega$, so we may suppress $\Omega$ in our notation.

We write 
\begin{equation}
\label{eq:sigma1}
\Sigma^1(G)=\{\chi\in S(G)\mid \cay_{\chi}(G)\textrm{ is connected}\},
\end{equation}
and $E^1(G)$ for the complement of $\Sigma^1(G)$.  If $N$ is a normal subgroup 
of $G$, we write $S(G,N)$ for the real characters in $S(G)$ which vanish on $N$.  
The following result is fundamental in BNS theory:

\begin{thm}[\cite{bns}]
\label{t:bns}
Let $G$ be a finitely generated group, and let $G/N$ be an infinite abelian quotient.  
The group $N$ is finitely generated if and only if $S(G,N)\subset \Sigma^1(G)$.  
In particular, $G'$ is finitely generated if and only if $E^1(G)=\emptyset$.
\end{thm}

By analogy to a result of Beauville on the structure of K\"ahler groups, 
we have the following result:

\begin{thm}
\label{t:beauville}
Let $G$ be a finitely generated group which is RFR$p$ for infinitely 
many primes $p$.  If $E^1(G)=\emptyset$ (or, 
if $E^1(G/G^k)=\emptyset$, for some $k\ge 2$), then $G$ is abelian.
\end{thm}

\begin{proof}
This follows immediately from Theorem \ref{t:rfrplarge} and Corollary \ref{c:rfrpfg}.
\end{proof}

\section{A combination theorem for RFR$p$ groups}
\label{sect:graph of groups}
In this section, we wish to give suitable hypotheses on vertex spaces, 
edge spaces, and gluing maps in a graph of spaces which guarantee 
that the resulting space has an RFR$p$ fundamental group. The hypotheses 
in Theorem \ref{thm:combination} may be difficult to verify in general, 
though we will show that within a certain natural class of graphs of 
spaces, the hypotheses are satisfied.

\subsection{Graphs of spaces}
\label{subsect:graph of spaces}
Let $\gam$ be a finite graph with vertex set $V(\gam)$ and edge set 
$E(\gam)\subset V(\gam)\times V(\gam)$. To each vertex $v\in V(\gam)$, 
we associate a connected, finite CW-complex $X_v$. Let $e=\{s,t\}\in E(\gam)$ be 
an edge. To each such edge $e$ we associate a connected, finite CW-complex 
$X_e\times [0,1]$, together with maps of CW complexes 
$\phi_{e,s}\colon X_e\times\{0\}\to X_s$ and $\phi_{e,t}\colon X_e\times\{1\}\to X_t$. 

We build the \emph{graph of spaces} $X_{\gam}$ by identifying $X_e\times\{0\}$ 
and $X_e\times\{1\}$ with their images under $\phi_{e,s}$ and $\phi_{e,t}$, 
respectively, for each edge of $\gam$. Replacing the discussion of CW-complexes 
with groups, we obtain a \emph{graph of groups}. Note that in the most general definition 
of a graph of spaces, we do not assume that $\gam$ is a simplicial graph, 
nor that the maps $\{\phi_{e,v}\}_{e\in E(\gam),v\in V(\gam)}$ induce 
injective maps on fundamental groups.

When considering graphs of groups, some care is required in choosing basepoints.
Specifically, one needs to choose identifications of vertex groups as subgroups of
the graph of groups, and of edge groups as subgroups of vertex groups. In the sequel,
we will always assume that such identifications have been made. For the most part,
any group theoretic verifications that need to be made are robust under the action of
the ambient group by conjugation.

If $Y\to X=X_{\gam}$ is a finite covering space, we will implicitly pull back 
the graph of spaces structure on $X$ to $Y$. In particular, the vertex 
spaces of $Y$ are the components of the preimages of the vertex 
spaces of $X$, and the edge spaces of $Y$ are the components 
of the preimages of the edge spaces of $X$.

Observe that a graph of spaces $X=X_{\gam}$ is equipped with a natural 
\emph{collapsing map}\/ $\kappa\colon X\to \gam$, which collapses each 
vertex space $X_v$ to a point and each thickened edge space $X_e\times [0,1]$ 
to the interval $[0,1]$. We choose an arbitrary splitting $\iota\colon \gam\to X$. 
For each vertex space $X_v$, we choose a basepoint 
$p_v\in \iota(\gam)\cap X_v$, which we identify with a basepoint for $\pi_1(X_v)$. 
For each free homotopy class of loops $\gamma\subset X$, we put $\gamma$ 
into \emph{standard form}. That is to say, $\gamma$ is allowed to trace out any 
based homotopy class of loops in $X_v$, based at $p_v$, and is allowed to 
travel between two adjacent vertex spaces along $\iota(\gam)$ only. It is clear 
that any homotopy class of loops in $X$ can be put into standard form.

Note that if $Y\to X$ is is a finite covering space, then the space $Y$ admits 
a natural collapsing map $\kappa_Y\colon Y\to \gam_Y$, and the graph 
$\gam_Y$ admits a natural map to $\gam=\gam_X$. These four maps form 
a natural commutative square, though it is important to note that 
$\gam_Y\to\gam$ is generally not a covering map.

If $\gamma\subset X$ is a homotopy class of loops, we define the 
\emph{combinatorial complexity}\/ of $\gamma$ to be the number of 
times $\gamma$ travels between two adjacent vertex spaces, minimized 
over all representatives of $\gamma$ which are in standard form. 
In other words, we count the number of times that $\gamma$ 
traverses an edge space of $X$. We write $C(\gamma)$ for 
the combinatorial complexity of $\gamma$.

If $\gamma\subset X$ is a homotopy class of loops, we define the 
\emph{backtracking number}\/ of $\gamma$ to be the number of 
times which $\gamma$ enters a vertex space $X_v$ through an 
edge space $X_e$, and then exits $X_v$ through the same edge 
space $X_e$, summed over all vertices and edge spaces. In other 
words, the backtracking number of $\gamma$ is the total number 
of times the combinatorial loop $\kappa(\gamma)$ backtracks inside 
$\gam$. If a vertex $v$ contributes to the backtracking number of 
$\gamma$, we will say that the loop $\gamma$ backtracks at the 
vertex space $X_v$. We write $B_X(\gamma)$ for the backtracking 
number of $\gamma$ in $X$.

Note that if $Y\to X$ is a covering space to which $\gamma$ lifts, then $\gamma$ 
and any of its lifts $\gamma'$ have backtracking numbers $B_X(\gamma)$ and $B_Y(\gamma')$. 
It is straightforward to see that $B_X(\gamma)\geq B_Y(\gamma')$. We may thus say 
that the backtracking number is non-increasing along covers. By convention, the 
backtracking number of a loop can only be positive if the combinatorial complexity 
of the loop is positive, so a loop of combinatorial complexity zero will have 
backtracking number zero.

\subsection{A combination theorem for RFR$p$ groups}
We now establish the main result of this section:

\begin{thm}
\label{thm:combination}
Fix a prime $p$. Let $X=X_{\gam}$ be a finite graph of connected, finite CW-complexes 
with vertex spaces $\{X_v\}_{v\in V(\gam)}$ and edge spaces $\{X_e\}_{e\in E(\gam)}$ 
satisfying the following conditions:
\begin{enumerate}
\item
For each $v\in V(\gam)$, the group $\pi_1(X_v)$ is RFR$p$.
\item
For each $v\in V(\gam)$, the RFR$p$ topology on $\pi_1(X)$ induces the 
RFR$p$ topology on $\pi_1(X_v)$ by restriction.
\item
For each $e\in E(\gam)$ and each $v\in e$, we have that the image 
\[
\phi_{e,v}(\pi_1(X_e))<\pi_1(X_v)
\] 
is closed in the RFR$p$ topology on $\pi_1(X_v)$.
\end{enumerate}
Then $\pi_1(X)$ is RFR$p$.
\end{thm}

The reader may compare the hypotheses of Theorem \ref{thm:combination} 
to the notion of \emph{$\mathfrak{P}$-efficiency} (see \cite{AF13} and \cite{WZ10}).

Observe that we do not assume that the gluing maps of the edge spaces to 
the vertex spaces induce injections on the level of fundamental groups, 
which could at least in principle have disastrous algebraic consequences. 
However, the assumption that each vertex space has an RFR$p$ fundamental 
group and that the RFR$p$ topology on $\pi_1(X)$ induces the RFR$p$ topology 
on the fundamental group of each vertex space implies that the inclusion $X_v\to X$ 
induces an injection on the level of fundamental groups.

The condition that the image of the edge space fundamental group is closed in the 
RFR$p$ topology on the fundamental group of the vertex space may seem difficult 
to verify, though we will show in the sequel that under natural hypotheses, this condition 
is automatically satisfied.  

Before proving the result, we note that we use very few specifics about the 
RFR$p$ topology. Indeed, the proof we give below could be suitably adapted 
to prove a combination theorem for the following classes of groups:

\begin{itemize}
\item
Residually finite groups;
\item
Residually solvable groups;
\item
Residually nilpotent groups;
\item
Residually $p$ groups.
\end{itemize}

\begin{proof}[Proof of Theorem \ref{thm:combination}]
We will prove the theorem by induction on the combinatorial 
complexity and the backtracking number, $(C(\gamma),B(\gamma))$, 
ordered lexicographically. We will denote the standard RFR$p$ tower 
of $X$ by $\{X_i\}_{i\geq 1}$, and we will show that for each nontrivial 
homotopy class of closed loops $\gamma\subset X$, 
there is some $i$ such that $\gamma$ does not lift to $X_i$.

For the base case, we suppose that $\gamma$ has combinatorial complexity zero, 
so that $\gamma$ remains inside of a single vertex space $X_v$ for its entire itinerary. 
Viewing $\gamma$ as a based homotopy class of loops, we identify $\gamma$ 
with an element of $G_v=\pi_1(X_v)$. We write $\{G_{v,i}\}_{i\geq 1}$ for the 
standard RFR$p$ filtration of $G_v$, so that $\gamma\in G_{v,i}\setminus G_{v,i+1}$ 
for some $i$. Similarly, we will write $G=\pi_1(X)$ and $\{G_i\}_{i\geq 1}$ 
for the standard RFR$p$ filtration on $G$.

By assumption we have that the RFR$p$ topology on $X$ induces the 
RFR$p$ topology on $X_v$. Thus for each $i$, there is a $j$ such that 
$G_j\cap G_v<G_{v,i}$. Since $\gamma\in G_{v,i}\setminus G_{v,i+1}$ 
for some $i$, we have that $\gamma_i\in G_j\setminus G_{j+1}$ for some $j$, 
which establishes the base case of the induction.

We now assume that the combinatorial complexity of $\gamma$ is $n>0$. 
We may suppose for a contradiction that $\gamma$ lifts to a loop 
$\gamma_i\subset X_i$ for each $i$. Writing 
$\kappa_i\colon X_i\to\gam_i$ for the collapsing map induced by 
pulling back the graph manifold structure of $X$ to $X_i$, we 
may assume that $\kappa_i(\gamma_i)$ is nullhomotopic in 
$\gam_i$ for each $i$ and each lift $\gamma_i$ of $\gamma$ 
to $X_i$. Indeed, otherwise observe that $\pi_1(X_i)$ surjects to 
$\pi_1(\gam_i)$ via $\kappa_i$, the latter of which is an RFR$p$ group. 
By Proposition \ref{prop:radical}, we have that any element of 
$\pi_1(X_i)$ which is not in $\ker\kappa_i$ does not lie in 
$\rad_p(\pi_1(X_i))=\rad_p(\pi_1(X))$.

For a lift $\gamma_i\subset X_i$ of $\gamma$, we will write $B_i(\gamma_i)$ 
for its backtracking number. Note that because the cover $X_i\to X$ is regular, 
this number is independent of the choice of lift. Furthermore, the combinatorial 
complexity $C(\gamma_i)$ is constant under passing to covers, so we will just 
denote it by $n$. Observe that we have the \emph{a priori} estimate 
$B_i(\gamma_i)\leq n/2$ for all $i$. Furthermore, if $\gamma$ lifts to 
each cover $X_i$ of $X$, then we must have $B_i(\gamma_i)\geq 1$ for all $i$. 
This is simply because a nullhomotopic loop in a graph must backtrack at least 
once. Thus, we have that the pair $(C(\gamma_i),B_i(\gamma_i))$ is bounded 
below by $(n,1)$ for all $i$.

We claim that if $\gamma$ lifts to each $X_i$, then for any loop $\gamma_i$ 
with associated data 
\begin{equation}
\label{eq:cb gamma}
(C(\gamma_i),B(\gamma_i))=(n,B(\gamma_i)),
\end{equation}
we can either decrease $B(\gamma_i)$ by one or we can decrease 
$C(\gamma_i)$ by two, after passing to a sufficiently high index $i$. 
This will prove the result by completing the induction.

For each cover $X_i$ in the standard RFR$p$ tower of $X$, we will 
fix a lift of $\gamma$, say $\gamma_i$. We choose these lifts at the 
beginning so that if $k>i$, then the cover $X_k\to X_i$ restricts to a 
map $\gamma_k\to \gamma_i$. Since each lift $\gamma_i$ has 
combinatorial length exactly $n$, we write 
\begin{equation}
\label{eq:gamma curves}
\{\gamma_i^1,\ldots,\gamma_i^n\}
\end{equation}
for the segments which are the intersections of $\gamma_i$ with the vertex 
spaces of $X_i$, i.e., 
\begin{equation}
\label{eq:gamma ij}
\gamma_i^j=\gamma_i\cap X_i^j\subset X_i.
\end{equation}
We will label these segments (and \emph{ipso facto}\/ 
the corresponding vertex spaces) coherently, so that for $k>i$, the 
segment $\gamma_i^j$ is covered by the segment $\gamma_k^j$. 
For each $i,j$, we have that $\gamma_i^j$ and $\gamma_i^{j+1}$ 
lie in different vertex spaces of $X_i$, by definition.

However, since $\kappa_i(\gamma_i)$ 
is nullhomotopic in $\gam_i$, we have that for some $j$, the segments 
$\gamma_i^{j-1}$ and $\gamma_i^{j+1}$ lie in the same vertex space 
\begin{equation}
\label{eq:xii}
X_i^{j-1}=X_i^{j+1}\subset X_i,
\end{equation}
and that the segments $\gamma_i^{j-1}$ and $\gamma_i^{j+1}$ meet 
the segment $\gamma_i^j\subset X_i^j$ in the same edge space $Y_i$. 
In other words, $\gamma_i$ backtracks at $X_i^j$.

We claim that if the loop $\gamma_i$ backtracks at the vertex space 
$X_i^j$ for all $i$, then either we may deform the segment $\gamma_i^j$ 
into $X_i^{j-1}$ for $i\gg1$, or we may decrease $B_i(\gamma_i)$ by at 
least one, for $i\gg 1$. Note that in the first case, it follows then that 
we have decreased the combinatorial complexity of $\gamma$ by at 
least two (after passing to a sufficiently high index $i$), so that this 
will complete the induction.

For each $i$, let us again write $Y_i$ for the edge space of $X_i$ between 
$X_i^{j-1}$ and $X_i^j$. Fixing a basepoint $y_i$ in each $Y_i$, we may 
deform $\gamma_i^j$ to be a closed loop in $X_i^j$ which is based at $y_i$, 
because $\gamma_i$ enters and exits $X_i^j$ through the same edge space 
$Y_i$. We will denote this based loop inside of $X_i^j$ by $\beta_i$. 
Of course, $\beta_i$ is well-defined only up to an element of $\pi_1(Y_i,y_i)$.

Observe that by the minimality of the combinatorial complexity of $\gamma$, 
we may assume that it is not possible to deform $\beta_i$ into $Y_i$ for any $i$, 
because then $\gamma_i$ could be pushed to avoid $X_i^j$ entirely, decreasing 
the combinatorial complexity by two and completing the induction. Fixing $i$, 
we claim that for $k\gg i$, each component $\beta_k$ of the preimage of the 
loop $\beta_i$ in the vertex space $X_k^j\subset X_k$ will be an arc traversing 
two distinct edge spaces of $X_k$. In particular, the loop $\gamma_k$ no 
longer backtracks at the vertex space $X_k^j$ (see Figure \ref{fig:backtrack}).

\begin{figure} 
  \includegraphics[width=\textwidth]{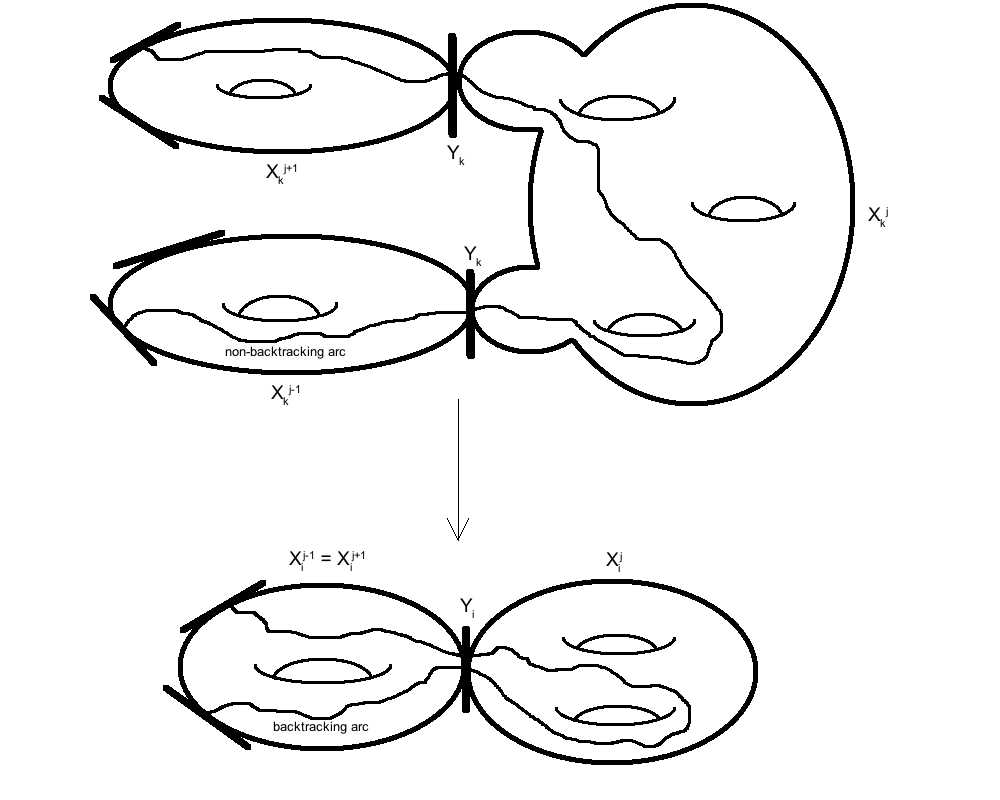}
\caption{Lifting a backtracking arc to a non-backtracking arc}
\label{fig:backtrack}
\end{figure}

To see this last claim (and thus establish the theorem), we need to find a 
$k\gg i$ so that the deck transformation of the cover $X_k^j\to X_i^j$ 
corresponding to the image of the homotopy class of the loop $\beta_i$ 
does not lie in the image of the subgroup $\phi_*(\pi_1(Y_i))$, where 
here we abuse notation and write $\phi$ for the gluing map which 
attaches $Y_i$ to $X_i^j$.

By assumption, the subgroup $\phi_*(\pi_1(Y_i))$ is closed in the RFR$p$ 
topology on $\pi_1(X_i^j)$, and the RFR$p$ topology on $\pi_1(X_i^j)$ 
coincides with the restriction of the RFR$p$ topology on $\pi_1(X_i)$. 
Writing $G_{i,k}$ for the deck group of $X_k^j\to X_i^j$, the statement that 
the subgroup $\phi_*(\pi_1(Y_i))$ is closed in the RFR$p$ topology on 
$\pi_1(X_i^j)$ is exactly the statement that for every element 
\begin{equation}
\label{eq:g element}
g\in \pi_1(X_i^j) \setminus \phi_*(\pi_1(Y_i)),
\end{equation}
there is a $k$ such that the image of $g$ in $G_{i,k}$ does not lie in the image 
of $\phi_*(\pi_1(Y_i))$. Thus, for $k\gg i$, we have that the image of the homotopy 
class of $\beta_i$ in $G_{i,k}$ does not lie in the image of $\phi_*(\pi_1(Y_i))$.

It follows that if $\beta_k$ is a component of the preimage of $\beta_i$ in such 
a cover $X_k^j\to X_i^j$, then the endpoints of $\beta_k$ cannot both lie in a 
single component of the preimage of $Y_i$. This establishes the claim and proves 
the theorem.
\end{proof}

\subsection{Applications of the combination theorem}
\label{subset:apply comb}

We now give the promised missing part of Theorem \ref{t:rfrp groups}:
\begin{cor}
\label{cor:free product}
The class of RFR$p$ groups is closed under taking 
finite free products.
\end{cor}

\begin{proof}
By induction, it suffices to prove the corollary for two RFR$p$ 
groups. Let $G=\pi_1(X)$ and $H=\pi_1(Y)$ be two such groups, 
where $X$ and $Y$ are connected, finite, based CW-complexes. 
The one-point union $Z=X\vee Y$ is homotopic 
to a CW-complex which has the structure of a graph of spaces, where 
$X$ and $Y$ are the vertex spaces and where the edge space is a point.

By assumption, $G$ and $H$ are RFR$p$ groups; hence, the trivial group is 
closed in the RFR$p$ topology on both $G$ and $H$. It therefore suffices 
to prove that the RFR$p$ topology on $G*H$ restricts to the RFR$p$ topology 
on $G$ and on $H$.

This last claim is straightforward, though. Setting $\{X_i\}_{i\geq 1}$, $\{Y_i\}_{i\geq 1}$, 
and $\{Z_i\}_{i\geq 1}$ to be the standard RFR$p$ towers for $X$, $Y$, and $Z$ respectively, 
we have that for each $i$, the space $Z_i$ is homotopy equivalent to a wedge of circles, 
glued along one point to a finite collection of CW-complexes, each of which is homotopy 
equivalent to either $X_i$ or $Y_i$. The integral first homology of $Z_i$ is just the direct 
sum of the integral first homologies of these spaces. It follows easily then that the RFR$p$ 
topology on $G*H$ restricts to the RFR$p$ topology on both $G$ and $H$.

The corollary now follows by Theorem \ref{thm:combination}.
\end{proof}

\begin{comment}
Note that the statement of Theorem \ref{thm:combination} can be modified slightly in 
order to obtain the following negative combination theorem, thus elucidating the 
necessity of some of the hypotheses in Theorem \ref{thm:combination}:
\begin{prop}
\label{prop:not rfrp}
Let $X=X_{\gam}$ be a finite graph of finite spaces with vertex spaces 
$\{X_v\}_{v\in V(\gam)}$, and let $v\in V(\gam)$ be a vertex such that 
the following conditions are satisfied:
\begin{enumerate}
\item
The group $\pi_1(X_v)$ is not RFR$p$.
\item
The RFR$p$ topology on $\pi_1(X)$ induces the RFR$p$ topology on 
$X_v$ by restriction.
\end{enumerate}
Then $\pi_1(X)$ is not RFR$p$.
\end{prop}

\begin{proof}
Since the RFR$p$ topology on $\pi_1(X)$ induces the RFR$p$ topology on 
$\pi_1(X_v)$ by the second assumption, we have that 
\[
\rad_p(\pi_1(X_v))<\rad_p(X)
\] 
for that $v$. But the first assumption implies that $\rad_p(\pi_1(X_v))\neq \{1\}$, 
so that 
\[
\rad_p(\pi_1(X))\neq \{1\},
\] 
i.e., $\pi_1(X)$ is not RFR$p$.
\end{proof}
\end{comment}

It is also possible to use Theorem \ref{thm:combination} to prove that 
right-angled Artin groups enjoy the RFR$p$ property, which is 
part \eqref{raag} of Proposition \ref{prop:ex rfrp} from the introduction.

\begin{prop}
\label{prop:raag rfrp}
Right-angled Artin groups are RFR$p$, for all primes $p$. 
\end{prop}

\begin{proof}
First note that we may as well consider the case where the defining 
graph $\gam$ is connected, since, as we just showed, 
the class of RFR$p$ groups is closed under finite 
free products.  The essential point is that for each 
$v\in V(\gam)$ one has a graph of groups decomposition 
\begin{equation}
\label{eq:RAAGdec}
A(\gam)\cong A(\gam_v)*_{A(\lk(v))} A(\st(v)),
\end{equation}
where $\gam_v$ is the subgraph of $\gam$ spanned by $V(\gam)\setminus \{v\}$, 
and where $\st(v)$ and $\lk(v)$ are the star and link of $v$, respectively. 

The RFR$p$ topology on $A(\gam)$ induces the RFR$p$ topology on both 
vertex groups by induction on $\abs{V(\gam)}$. We have that $\gam_v$ is a 
proper subgraph of $\gam$ and $A(\gam)$ retracts to $A(\gam_v)$, so we 
can apply Corollary \ref{cor:retract}. The right-angled Artin group $A(\st(v))$ 
is the direct product of $\Z$ with $A(\lk(v))$, both of which 
are retracts of $A(\gam)$, so that Corollary \ref{cor:retract} applies again.

Similarly by induction on $\abs{V(\gam)}$, both vertex groups 
are RFR$p$. The final verification needed before applying 
Theorem \ref{thm:combination} is to show that the edge group $A(\lk(v))$ 
is closed in the RFR$p$ topology on each vertex space. Since $A(\lk(v))$ is 
a retract of $A(\gam)$, we apply Corollary \ref{cor:closed retract} 
to confirm that fact.

The result now follows from Theorem \ref{thm:combination}.
\end{proof}

\begin{remark}
\label{rem:hw}
An alternative argument can be given, based on Lemma 3.9 from \cite{HW99}.
\end{remark}

The proof of Proposition \ref{prop:raag rfrp} 
has  the following immediate corollary.

\begin{cor}
\label{cor:closed raag}
Let $\Lambda<\gam$ be a (full) subgraph and let $p$ be a prime. Then 
$A(\Lambda)<A(\gam)$ is closed in the RFR$p$ topology.
\end{cor}

\section{$3$-manifold groups and the RFR$p$ property}
\label{sect:3manifold}

\subsection{Fundamental groups of $3$-manifolds}
\label{subsect:3mfd groups}
A group $G$ is called a \emph{$3$-manifold group}\/ if it can be realized as 
the fundamental group of a compact, connected, orientable $3$-manifold $M$ 
with $\chi(M)=0$. In this section, we study $3$-manifold groups and whether or 
not they are RFR$p$, for both geometric manifolds and non-geometric manifolds. 
In the first case, we can exactly characterize which geometric $3$-mani\-folds 
groups are virtually RFR$p$, and in the second case we can give some 
hypotheses which guarantee that a non-geometric $3$-manifold group is RFR$p$. 

The reader will note that the hypotheses we place on the non-geometric $3$-mani\-fold 
groups are modeled on boundary manifolds of curve arrangements in $\C^2$, and 
indeed in this section we will prove that such a boundary manifold has an RFR$p$ 
fundamental group.

We will restrict our discussion to \emph{prime}\/ $3$-manifolds, namely 
ones which cannot be decomposed as nontrivial connected sums. Note that on the level 
of fundamental groups, connected sum corresponds to free product, and the free product 
of two finitely generated groups will be RFR$p$ if and only if both free factors are RFR$p$ 
(cf. Theorem \ref{t:rfrp groups} and Corollary \ref{cor:free product}).

\subsection{Geometric $3$-manifolds}
\label{subsect:geometric}
Recall that a $3$-manifold $M$ is \emph{geometric}\/ if it admits 
a finite volume complete metric modeled on one of the eight Thurston 
geometries,
\[
\{S^3,S^2\times\R, \R^3,\Nil,\Sol,\bH^2\times \R,
\widetilde{\PSL_2(\R)},\bH^3\},
\]
see \cite{Perelman1, Perelman2, Scottgeom, ThurstonBook}. 
Perelman's Geometrization Theorem says every prime $3$-manifold can be cut up along a 
canonical collection of incompressible tori into finitely many pieces, every one of which is 
geometric. It is well-known (see \cite{ThurstonBook}) that if a manifold is geometric, then 
its geometry can be read off from the structure of its fundamental group, and conversely 
the geometry of a $3$-manifold determines the structure of its fundamental group. 
We are therefore prepared to give a proof of Theorem \ref{thm:geometric} as 
claimed in the introduction.

Recall that Theorem \ref{thm:geometric} asserts that certain geometric 
$3$-manifold groups $G$ are virtually RFR$p$, but not necessarily 
RFR$p$. This is an important distinction. For one, if we allow for 
orbifolds and orbifold fundamental groups, then $G$ could potentially 
have torsion and therefore not be RFR$p$ for any prime $p$. More 
essentially, there are geometric $3$-manifold groups which fail to 
be RFR$p$ for any prime, but which become RFR$p$ for every prime 
after passing to a finite index subgroup. We illustrate this assertion 
with a class of examples.

\begin{example}
\label{ex:hyp knot}
Let $G$ be the fundamental group of a hyperbolic knot complement.  
Then $G$ falls under the purview of Theorem \ref{thm:geometric}, 
so that there is a finite index subgroup $K<G$ such that $K$ is 
RFR$p$ for every prime $p$. On the other hand, it is an easy 
exercise to check that for each prime $p$ we have that 
$\rad_p(G)=[G,G]\neq \{1\}$, since a nonabelian $p$-group 
must have noncyclic abelianization.
\end{example}

\begin{proof}[Proof of Theorem \ref{thm:geometric}]
Let $G=\pi_1(M)$ be a geometric $3$-manifold group. We begin with the geometries 
$\{S^3,S^2\times\R,\R^3\}$. In the case of $S^3$, we have that $G$ is finite 
and so there is nothing to show. In the other two cases, $G$ either contains 
$\Z$ or $\Z^3$ with finite index, in which case it is clear that $G$ is RFR$p$ for 
each prime.

If $M$ is modeled on $\bH^2\times\R$, then a finite index subgroup of $G$ is 
isomorphic to $\pi_1(S)\times\Z$, where $S$ is an orientable surface. 
Combining Proposition \ref{prop:ex rfrp} and Theorem \ref{t:rfrp groups}, we 
have that $G$ is virtually RFR$p$ for every prime.

If $M$ is modeled on $\bH^3$, then Agol's resolution of the Virtual Haken Conjecture \cite{agolgm} 
shows that a finite index subgroup of $G$ lies as a finitely generated subgroup of a right-angled 
Artin group and is therefore RFR$p$ for every prime, by Theorem \ref{t:rfrp groups}.

If $M$ is modeled in $\Nil$ geometry, then every finite index subgroup of $G$ is nonabelian 
and nilpotent, and hence not RFR$p$ for any prime $p$ by Proposition \ref{prop:rfrp rtfn}.

If $M$ is modeled on the $\Sol$ geometry, then $G$ has a finite index subgroup 
$H$ which is a semidirect product of $\Z^2$ with $\Z$, where the conjugation 
action of $\Z^2$ is by a hyperbolic matrix. Any finite index subgroup $K$ of 
$H$ has rank one abelianization, so that $\rad_p(K)\neq \{1\}$ for any $p$.

Finally, if $M$ is modeled on $\widetilde{\PSL_2(\R)}$, then a finite index subgroup of 
$G$ is a nonsplit central extension of $\pi_1(S)$ by $\Z$, where $S$ is a closed, orientable 
surface of genus at least two. By Theorem \ref{thm:circle bundle}, we have that $G$ is not 
virtually RFR$p$ for any prime $p$.  This completes the proof.
\end{proof}

\subsection{Graph manifolds}
\label{subsect:graph}
We wish to develop criteria which allow one to verify the hypotheses 
of Theorem \ref{thm:combination}, and thus prove that certain 
non-geometric $3$-manifold groups are RFR$p$, and deduce 
Theorem \ref{thm:bdyrfrp}. For our purposes, a prime 
$3$-manifold $M$ is an \emph{graph manifold}\/ 
if it is a graph of spaces $X$ satisfying the following conditions:
\begin{enumerate}
\item \label{gr1}
Each vertex space $X_v$ is a Seifert-fibered manifold, with $\deg(v)$ being at 
most the number of components of $\partial X_v$.
\item \label{gr2}
Each edge space $X_e$ is a torus.
\item \label{gr3}
The gluing maps which assemble $X$ are given by matching the two 
boundary components of $X_e\times[0,1]$ via an orientation-preserving 
homeomorphism to a component of $\partial X_v$ and $\partial X_{w}$ 
respectively, where $e=\{v,w\}$.
\end{enumerate}

For a general graph manifold $M$ as we have defined it here, it may not be the case 
that $\pi_1(M)$ is RFR$p$, even if each of the vertex manifolds have RFR$p$ fundamental 
groups. We illustrate this phenomenon with a family of examples. 

\begin{example}
\label{ex:torus knots}
Let $M_1$ and $M_2$ be torus knot complements in $S^3$, 
which are well-known to be Seifert-fibered. We will write $K_1$ and $K_2$ 
for the respective fundamental groups. 
The cusps of $M_1$ and $M_2$ give rise to copies of $\Z^2$ inside of $K_1$ and 
$K_2$ respectively, and on the level of homology, the maps $\phi_i\colon \Z^2\to H_1(M_i,\Z)$ 
induced by inclusion have rank one. So, inside of $K_i$, we will decompose $\Z^2$ 
as a direct sum of two cyclic groups $A_i\oplus B_i$, where $B_i=\ker \phi_i$. 

Let us glue now $M_1$ to $M_2$ along the cusps to get a new graph 
manifold $M$, in such a way that $A_1$ is identified with $B_2$ and $A_2$ 
is identified with $B_1$.   The resulting $3$-manifold has trivial first 
homology, and so does not have an RFR$p$ fundamental group. In terms 
of Theorem \ref{thm:combination}, we see that the RFR$p$ topology on 
$\pi_1(M)$ does not induce the RFR$p$ topology on either $K_1$ or $K_2$.
\end{example}

\subsection{The vertex manifolds}
\label{subsec:class}

We now proceed with the construction of the graph manifolds 
comprising the class $\mX$ described in the introduction. We 
go over each of the five axioms, and introduce some further notation 
and terminology along the way.  

 \ref{x1}
Let $\gam$ be a finite, connected, bipartite, simplicial graph such that each 
vertex has degree at least two.  We color the vertices two colors, and we 
denote the resulting equivalence classes by $\mL$ and $\mP$.  If 
$v=L\in\mL$, we write $\mP_L\subset\mP$ for the 
set of vertices which are adjacent to $L$, and similarly if $v=P\in\mP$, 
we write $\mL_P\subset\mL$ for the set of vertices which are adjacent to $P$.

\ref{x2}
We build a graph manifold $X=X_{\gam}$ whose underlying graph is 
$\gam$ as follows.  For each vertex $v\in V(\gam)$, the vertex manifold 
$X_v$ is homeomorphic to $S^1\times S_v$, where $S^1$ is the circle, 
and where 
\begin{equation}
\label{eq:surf v}
S_v=S\setminus\bigcup_{i=1}^{m(v)} D_i^2,
\end{equation}
where $S$ is a closed, orientable surface, and where $\{D^2\}_{i=1}^{m(v)}$ 
denotes a disjoint union of $m(v)$ open disks.  Thus, $\pi_1(X_v)\cong \bZ\times F_{k(v)}$, 
where $F_{k(v)}$ denotes the free group of rank $k(v)$, a number which 
depends on $m(v)$ and on the genus of $S$.  As part of axioms \ref{x3} and 
\ref{x4}, will make the following assumptions on the graph $\gam$:
\begin{enumerate}[label=($\mX'_\arabic*$),leftmargin=4\parindent]
\setcounter{enumi}{2}
\item \label{mv1}
If $v=L\in\mL$, then $m(v)\geq\deg v+1$.
\\[-10pt]
\item \label{mv2}
If $v=P\in\mP$, then $m(v)=\deg v$.
\end{enumerate}

In the first case, the boundary components of $S_v$ will be 
denoted by
\begin{equation}
\label{eq:clp bdry}
\{C_{L,P}\}_{P\in\mP_L}\cup\{C^1_{L},C^2_{L},\ldots,C^{r(L)}_{L}\}, 
\end{equation}
where $C_{L,P}$ corresponds to the edge $\{L,P\}$, and 
where $r(L)=m(L)-\deg(L)$.  In the second case, the boundary 
components of $S_v$ will be denoted by
\begin{equation}
\label{eq:cpl bdry}
\{C_{P,L}\}_{L\in\mL_P}, 
\end{equation}
where $C_{P,L}$ corresponds to the edge $\{P,L\}$.
The homology class of $C_{L,P}$ and $C_{P,L}$ in $H_1(S_v,\Z)$ 
will be written $b_{L,P}$ and $b_{P,L}$ respectively, and the homology 
classes of $\{C^{i}_{L}\}$ will be written $\{b^{i}_{L}\}$. 

\subsection{Euler numbers}
\label{subsec:framings}

Next, we explain the role played by the Euler numbers in axioms \ref{x3} and 
\ref{x4}. We start with a simple, motivating example. 

\begin{example}
\label{ex:hopf}
Suppose $X$ is the exterior of an two-component Hopf link in $S^3$.  
Then $X$ fibers overs the circle, with fiber an annulus, and with 
monodromy a Dehn twist around the core of the annulus.  Alternatively, 
the Hopf fibration $S^3 \to S^2$ restricts to a fibration of $X$ over $S^1\times I$. 
Since the Hopf fibration has Euler class one, the Euler number of the Seifert 
manifold $X$ is also one. 

It follows that $X$ is homeomorphic to a circle bundle over 
the annulus, but there is no trivialization preserving the annulus. 
We have that the group $H_1(X,\Z)$ is generated by the homology class  
$t$ of the fiber, together with the two boundary homology classes 
$b_1,b_2$ of the annulus, with the relation $t=b_1+b_2$ corresponding 
to an Euler number of one. The homology classes $b_1$ and $b_2$ 
of the boundary components of the annulus are homologous in the annulus itself, 
but not in $X$.  
\end{example}

We now return to an arbitrary graph manifold $X\in \mX$. 
By assumption, each vertex manifold $X_v$ is a Seifert manifold, 
since it is a circle bundle over an orientable surface 
with boundary with trivial monodromy, and thus homeomorphic 
to a product bundle.  Yet, as we saw in the above example, 
the trivialization may not preserve the underlying surface, 
and the discussion of Euler numbers below reflects this fact.

We will write $t_v$ for the homology generator of the $S^1$ factor of $X_v$, 
and $B_V$ for the total homology span of the boundary components 
of $\{S_v\}_{v\in V(\gam)}$ inside of $X_v$.  
We then have that $H_1(X_v,\bZ)$ is a free abelian group, and the 
homology classes in $\langle t_v,B_v\rangle$ satisfy a single 
linear relation determined by the \emph{Euler number}\/ 
of the corresponding Seifert manifold.  By definition, we 
take this Euler number, $e(v)$, to be the coefficient of 
$t_v$ in this linear relation.  In the terminology of Luecke 
and Wu \cite{LW}, the integer $e(v)$ is the relative Euler number 
of $X_v$ with respect to the chosen framing of $\partial X_v$, 
to wit, the curves in $\partial X_v$ corresponding to the curves in 
$\partial S_v$ specified in \eqref{eq:clp bdry} and \eqref{eq:cpl bdry} 
under the homeomorphism $X_v \cong S^1\times S_v$.   
In turn, this number coincides with the (orbifold) Euler number of 
the (closed) Seifert manifold obtained by filling in the boundary 
tori of $X_v$ with solid tori, while matching the framing 
of $\partial X_v$ with the meridians of these solid tori.

As the second part of axioms \ref{x3} and 
\ref{x4}, we will make the following assumptions on the 
integers $e(v)$.  

\begin{enumerate}[label=($\mX''_\arabic*$),leftmargin=4\parindent]
\setcounter{enumi}{2}
\item \label{ev1}
If $v=L$, we will assume that $e(v)=0$, 
so that the relation reads 
\begin{equation}
\label{eq:Euler0}
\sum_{P\in\mP_L} b_{L,P}+\sum_{j=1}^{r(L)}b^j_{L}=0.
\end{equation}
\\[-10pt]
\item \label{ev2}
If $v=P$, we will assume that $e(v)\ne 0$, 
so that the relation reads 
\begin{equation}
\label{eq:Euler1}
\sum_{L\in\mL_P} b_{P,L} = k_P\cdot t_P,
\end{equation}
where $k_P=e(P)$ is a nonzero integer.
\end{enumerate}

\subsection{Gluing maps}
\label{subsec:flips}

Finally, we need to define the gluing maps which allow us to 
assemble our class $\mX$ of graph manifolds.  If $e=\{L,P\}$ forms an edge in 
$\gam$, axiom \ref{x5} requires that we glue $X_L$ to $X_P$ via a \emph{flip map}.  
That is to say, we choose homeomorphisms $\psi_e\colon S^1\to C_{P,L}$ and 
$\bar{\psi}_e\colon C_{L,P}\to S^1$, and we glue $X_L$ to $X_P$ 
along $X_e\cong S^1\times S^1$ via the homeomorphism 
\begin{equation}
\label{eq:psi flip}
\psi_e\times\bar{\psi}_e\colon S^1\times C_{L,P}\to C_{P,L}\times S^1.
\end{equation}

Put another way, the chosen homeomorphisms of the vertex manifolds 
$X_v$ with $S^1\times S_v$ determine meridian-longitude pairs on each 
torus $X_{v,e}$, for each edge $e$ incident to $v$.  Given an edge $e=\{v,w\}$, 
we glue $X_{v}$ to $X_{w}$ by identifying the tori  $X_{v,e}$ and $X_{w,e}$ 
via a diffeomorphism represented in the aforementioned basis by the matrix 
$J=\left(\begin{smallmatrix} 0 & 1\\1 & 0\end{smallmatrix}\right)$.
For more information on this procedure, we refer to \cite{DoigHorn}.

In this way, we have assembled a compact, connected, orientable graph 
manifold $X=X_{\gam}$. This completes the description of our class $\mX$ 
of graph manifolds. 

\section{The boundary manifold of a plane algebraic curve}
\label{sect:boundary manifold}

Let $\mX$ be the class of all graph manifolds obtained by the procedure 
detailed in the previous section. Our next objective is to prove 
Theorem \ref{thm:graph manifold} from the introduction, which states that 
the fundamental groups of manifolds in this class enjoy the RFR$p$ property, 
for all primes $p$. 
Before proceeding with the proof, we motivate our result 
by showing that certain $3$-manifolds occurring in the topological 
study of complex plane algebraic curves belong to the class $\mX$.

\subsection{Projective algebraic curves}
\label{subsec:curves}

We refer the reader to~\cite{BK} for background on the material in 
this subsection. Let $\mC$ be an algebraic curve in the complex projective 
plane $\CP^2$, that is, the zero-locus of a homogeneous polynomial 
$f\in \C[x,y,z]$.  Without essential loss of generality, we may 
assume $\mC$ is reduced, i.e., $f$ has no repeated factors.  
By definition, the degree of $\mC$ is the degree of its defining 
polynomial $f$ (which is uniquely defined, up to constants).

Let $T$ be a regular neighborhood of $\mC$, and let $M_{\mC}=\partial T$ 
be its boundary.  Then $M_{\mC}$ is a closed, orientable $3$-manifold, 
called the {\em boundary manifold}\/ of the curve $\mC$.   
As shown by Durfee in \cite{Du83}, the homeomorphism type 
of $M_{\mC}$ is independent of the choices made in constructing the 
regular neighborhood $T$, and depends only on $\mC$. 

We will mainly be interested in the case when each irreducible 
component $C$ is smooth, and all the singularities of $\mC$ are 
simple, that is, any two distinct components intersect transversely. 
Here are a couple of well-known examples.

\begin{figure}
\centering
\begin{tikzpicture}[scale=0.55]
\draw[blue, style=thick] (-0.5,3) -- (2.5,-3);
\draw[blue, style=thick]  (0.5,3) -- (-2.5,-3);
\draw[blue, style=thick] (-3.5,-2) -- (3.5,-2);
\fill[red] (-2,-2) circle (4pt);  
\fill[red] (2,-2) circle (4pt);  
\fill[red] (0,2) circle (4pt);   
\node at (-4,-2) {$L_1$}; 
\node at (-2.7,-3.5) {$L_2$};
\node at (2.8,-3.5) {$L_3$};
\node at (-2.2,-1.1) {$P_1$}; 
\node at (2.3,-1.1) {$P_2$};
\node at (0.8,2.1) {$P_3$}; 
%%%%%%%%%%%%%%%%
\node [draw, thick, minimum size=3cm, regular polygon, regular polygon sides=6]  
(a) at (10,0) {};
\foreach \x in {1,3,5}
\fill[blue]  (a.corner \x) circle[radius=4pt];
\foreach \x in {2,4,6}
\fill[red] (a.corner \x) circle[radius=4pt];
\node[align=right, above] at (12,2.2) {\small $L_1$};
\node[align=right, below] at (12,-2.2) {\small $L_2$};
\node[align=left] at (6.6,0) {\small $L_3$};
\node[align=left, above] at (8,2.2) {\small $P_2$};
\node[align=left, below] at (8,-2.2) {\small $P_3$};
\node[align=right] at (13.4,0) {\small $P_1$};  
\end{tikzpicture}
%%%%%%%%%%%%%%%%%%
\caption{A generic arrangement of $3$ lines and its graph}
\label{fig:3lines}
\end{figure}
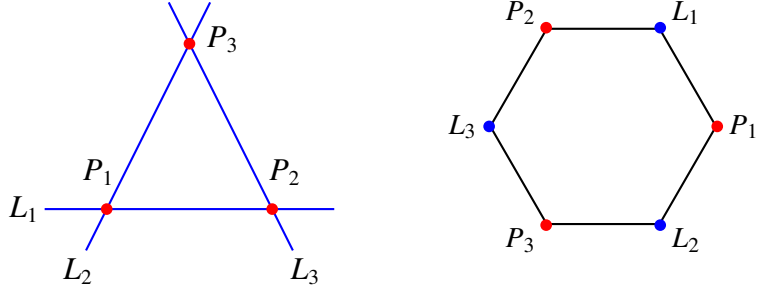

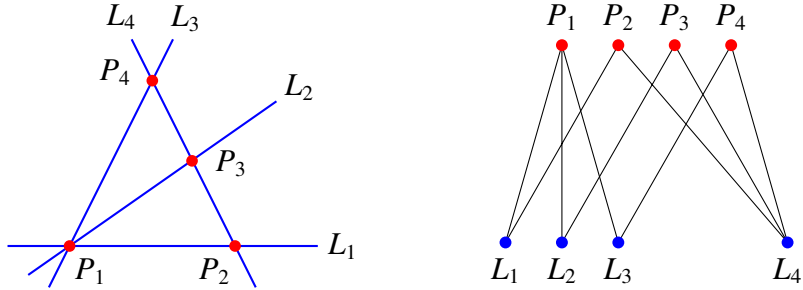
\begin{figure}
\centering
\begin{tikzpicture}[scale=0.55]
\draw[blue, style=thick] (-0.5,3) -- (2.5,-3);
\draw[blue, style=thick]  (0.5,3) -- (-2.5,-3);
\draw[blue, style=thick] (-3.5,-2) -- (4,-2);
\draw[blue, style=thick]  (-3,-2.7) -- (3,1.5);
\fill[red](-2,-2) circle (4pt);  \fill[red](2,-2) circle (4pt);  
\fill[red](0,2) circle (4pt);  \fill[red](0.96,0.06) circle (4pt);  
\node at (4.6,-2.0) {$L_1$}; \node at (3.55,1.85) {$L_2$};
\node at (0.8,3.5) {$L_3$}; \node at (-0.75,3.55) {$L_4$};
\node at (-1.5,-2.6) {$P_1$}; \node at (1.5,-2.6) {$P_2$}; 
\node at (1.9,0) {$P_3$}; \node at (-0.9,2.2) {$P_4$}; 
\end{tikzpicture}
\hspace*{0.5in}
\begin{tikzpicture}[scale=0.75]
\draw (0,0) -- (1,3.5) -- (1,0)--(3,3.5)--(5,0)--(2,3.5)--(0,0);  
\draw (1,3.5)--(2,0) --(4,3.5)--(5,0);
\fill[blue](0,0) circle (3pt);\fill[blue](1,0) circle (3pt);\fill[blue](2,0) circle (3pt);
\fill[blue](5,0) circle (3pt);\fill[red](1,3.5) circle (3pt);\fill[red](2,3.5) circle (3pt);
\fill[red](3,3.5) circle (3pt);\fill[red](4,3.5) circle (3pt);
\node at (0,-0.5) {$L_1$}; \node at (1,-0.5) {$L_2$};
\node at (2,-0.5) {$L_3$}; \node at (5,-0.5) {$L_4$};
\node at (1,4) {$P_1$}; \node at (2,4) {$P_2$};
\node at (3,4) {$P_3$}; \node at (4,4) {$P_4$};
\end{tikzpicture}
\caption{A near-pencil of lines and its associated graph}
\label{fig:nearpencil}
\end{figure}

\begin{example}
\label{ex:smooth curve}
Suppose $\mC$ has a single irreducible component $C$, which 
we assume to be smooth. Then $C$ is homeomorphic to an orientable 
surface $\Sigma_g$ of genus $g=\binom{d-1}{2}$, where $d$ is the degree 
of $C$. Moreover, by B\'{e}zout's theorem, $C\cdot C=d^2$.   Thus, $M_{\mC}$ is 
a circle bundle over $\Sigma_g$ with Euler number $e=d^2$. 
\end{example} 

\begin{example}
\label{ex:bdry pencil}
Suppose $\mC$ is a pencil of $n$ lines in $\CP^2$, 
defined by the polynomial $f=z_1^{n}-z_2^{n}$.  
Then $M_{\mC}=\sharp^{n-1} S^{1}\times S^{2}$; in particular, 
if $n=1$ (the case $d=1$ in the previous example), then $M_{\mC}=S^3$.  
\end{example}

\begin{example}
\label{ex:bdry near pencil}
Suppose $\mC$ is a near-pencil of $n$ lines in $\CP^2$, 
defined by the polynomial $f=z_1(z_2^{n-1}-z_3^{n-1})$, 
then $M_{\mC}=S^1\times\Sigma_{n-2}$.   The case $n=3$ (for 
which $M_{\mC}$ is the $3$-torus) is depicted in Figure \ref{fig:3lines}, 
while the case $n=4$ is depicted in Figure \ref{fig:nearpencil}.
\end{example} 

Note that in the first example the group $\pi_1(M_{\mC})$ is not RFR$p$, 
for any prime $p$, provided $d\ge 2$ (cf.~Theorem \ref{thm:circle bundle}), 
while in the second and third examples $\pi_1(M_{\mC})$ is RFR$p$ 
for all primes $p$. 
 
It turns out that the boundary manifold of a plane algebraic curve 
is a graph manifold.  This structure can be described in terms of 
Neumann's plumbing calculus \cite{Neu}.  We refer to \cite{DP,EN} 
for a detailed exposition of the subject, and to \cite{JY98, Hi01, CS08} 
for a more specific description in the case of line arrangements. Let us 
briefly review this construction, in the special context we consider here. 

Given an algebraic curve $\mC\subset \CP^2$, 
let $\L$ be the set of irreducible components, and 
let $\P$ be the set of multiple points, i.e., the set of points 
$P\in \CP^2$ where at least two distinct curves from $\L$ intersect.  
The underlying graph $\Gamma$ is the incidence graph of the 
arrangement of irreducible curves comprising $\L$. The 
graph has vertex set $\L \cup \P$, and has an edge joining 
$C$ to $P$ precisely when $C$ contains $P$ (see Figures 
\ref{fig:3lines}, \ref{fig:nearpencil}, and \ref{fig:circle-line}). 
The case when the curve $\mC$ is irreducible (and smooth) 
was treated in Example \ref{ex:smooth curve}.  
So let us assume that $\abs{\L}\ge 2$; in particular, $\abs{\P}\ge 1$,    
and the graph $\Gamma$ is bipartite.  For each point $P\in \mP$, 
the vertex manifold $M_P$ is the exterior of a Hopf link on as many  
components as the multiplicity of $P$.  Likewise, for each curve $C\in \L$, 
the vertex manifold $M_C$ is a circle bundle whose base is $C$ 
with a number of open $2$-disks, one for each multiple point lying on $C$.  
Finally, the vertex manifolds are glued my means of flip maps along 
boundary tori, as specified by the plumbing graph $\Gamma$, to 
produce the boundary manifold $M_{\mC}=M_{\Gamma}$.

\begin{figure}
\centering
\begin{tikzpicture}[scale=0.55]
\draw[blue, style=thick] (-4,0) -- (4,0);
\draw[blue, style=thick] (0,0) circle (1in);
\fill[red](-2.55,0) circle (4pt);  
\fill[red](2.55,0) circle (4pt); 
\node at (2,2.4) {\small $C$};
\node at (0,-0.5) {\small $L$};
\node at (-3,-0.5) {\small $P$};
\node at (3,-0.5) {\small $Q$}; 
\draw[style=thick]  (8,0) -- (10,2); 
\draw[style=thick]  (8,0) -- (10,-2); 
\draw[style=thick]  (12,0) -- (10,2); 
\draw[style=thick]  (12,0) -- (10,-2); 
\fill[blue](10,2) circle (4pt);  
\fill[blue](10,-2) circle (4pt); 
\fill[red](8,0) circle (4pt);  
\fill[red](12,0) circle (4pt);  
\node at (7.4,0) {\small $P$};
\node at (12.6,0) {\small $Q$};
\node at (10,2.7) {\small $C$};
\node at (10,-2.7) {\small $L$};
\end{tikzpicture}
\caption{A conic-line arrangement and its intersection graph}
\label{fig:circle-line}
\end{figure}
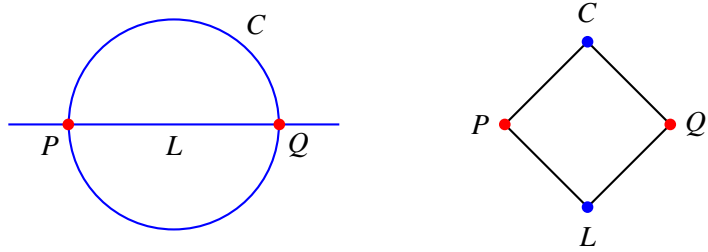

\begin{example}
\label{ex:heisenberg}
Suppose $\mC=C\cup L$ consists of a smooth conic and a transverse line, 
as in Figure \ref{fig:circle-line}. The graph $\Gamma$ is a square, and 
all vertex manifolds are thickened tori $S^1\times S^1 \times I$.  
Following the algorithm from  \cite[Theorem 5.1]{Neu}, as sketched 
in \cite[Figure 4.2]{DP}, we see that the boundary manifold $M_{\mC}$ 
is the mapping torus of a Dehn twist, or, alternatively, an $S^1$-bundle 
over $S^1\times S^1$ with Euler number $1$.  Either description shows 
that $M_{\mC}$ is the Heisenberg nilmanifold. In view of Proposition 
\ref{prop:rfrp rtfn} (or, alternatively, Theorem \ref{thm:circle bundle}), 
we conclude that the group $\pi_1(M_{\mC})$ is not RFR$p$, 
for any prime $p$. 
\end{example}  

\subsection{Affine algebraic curves}
\label{subsec:affine}

Similar considerations apply to affine, plane algebraic curves.  More 
precisely, let $\mC$ be a (reduced) algebraic curve in the 
affine plane $\C^2$, that is, the zero-locus of a polynomial 
$f\in \C[x,y]$ with no repeated factors.  As before, we will only 
consider the case when each irreducible component of $\mC$ is 
smooth, and all the singularities of $\mC$ are of \emph{type ${\rm A}$}, 
that is, their germs are isomorphic to a pencil of lines. Furthermore, 
we shall assume that each irreducible component of $\mC$ is 
transverse to the line at infinity.  

Let $\partial{T}$ be the boundary of a regular neighborhood of $\mC$.  
We define the {\em boundary manifold}\/ of the curve to be the intersection 
\begin{equation}
\label{eq:bdry affine}
M_{\mC}:=\partial T \cap B^4,
\end{equation}
where $B^4$ is a ball of sufficiently large radius, so that all singularities 
of $\mC$ are contained in this ball. Clearly, $M_{\mC}$ is a smooth, connected, 
orientable $3$-manifold, with boundary components tori $S^1\times S^1$ in   
one-to-one correspondence with the irreducible components of $\mC$.  

\begin{example}
\label{ex:affine smooth curve}
Suppose $\mC$ has a single (smooth) irreducible component of degree $d$.  
Let $\overline{M}$ be the $S^1$-bundle with Euler number $d^2$ over the Riemann 
surface of genus $\binom{d-1}{2}$ from Example \ref{ex:smooth curve}.  
The boundary manifold $M_{\mC}$, then, is obtained by removing open tubular 
neighborhoods of $d$ fibers of this bundle. Consequently,  
$\pi_1(M_{\mC})$ is isomorphic to $\Z\times F_{(d-1)^2}$, and thus 
is RFR$p$, for all primes $p$.
\end{example} 

We previously defined a class $\mX$ of compact 
graph manifolds $M$ for which the underlying graph 
$\gam$ is connected, bipartite, and each vertex in one of 
the parts has degree at least two, such that each vertex manifold  is 
a trivial circle bundle over an orientable surface with boundary 
obeying some technical conditions on the framings, and such that 
all the gluing maps are given by flips.  The main result of this 
section shows that all boundary manifolds arising from this  
construction belong to this class.

\begin{thm}
\label{thm:bdrycurve}
Let $\mC$ be a plane algebraic curve such that
\begin{enumerate}
\item \label{cc1}
Each irreducible component of $\mC$ is smooth and transverse to 
the line at infinity.
\item \label{cc2}
Each singular point of $\mC$ is a type ${\rm A}$ singularity.
\end{enumerate}
Then the boundary manifold $M_{\mC}$ lies in $\mX$. 
\end{thm}

\begin{proof}
We start by describing the underlying graph $\Gamma$ of the 
graph-manifold $M_{\Gamma}=M_{\mC}$.
Let $\L$ be the set of irreducible components of $\mC$, and 
let $\P$ be the set of multiple points of $\mC$, i.e., the set of points 
$P$ in $\C^2$ where at least two distinct curves from 
$\L$ meet.  By our assumption on the singularities of $\mC$,  
if $L_1$ and $L_2$ are two distinct 
components of $\mC$ meeting at a point $P\in \P$, 
then $L_1$ and  $L_2$ intersect transversely at $P$.  

The graph $\Gamma$ is the incidence graph of the resulting 
point--line configuration. This graph has vertex set 
$V(\Gamma)=\L\cup \P$ and edge set 
\begin{equation}
\label{eq:eglp}
E(\Gamma)=\{ (L,P) \in \L\times \P\mid P\in L\}.  
\end{equation}
(See Figures \ref{fig:3lines}, \ref{fig:nearpencil}, and \ref{fig:circle-line} 
for some illustrations.) 
Note that the link of a vertex $P$ is $\L_P=\{L \in \L \mid P\in L\}$, 
whereas the link of a vertex $L$ is $\P_L=\{P \in \P \mid P\in L\}$.  
In view of our assumptions, we have that $\abs{\L_P}\ge 2$ 
and $\abs{\P_L}\ge 1$, for all $P$ and $L$.  It follows that 
$\Gamma$ is a connected, bipartite graph, and each vertex 
$P\in \P$ has degree at least two. Thus, axiom \ref{x1} is satisfied 
by the graph $\Gamma$.  

By assumption, each component $L\in \mC$ is a smooth, irreducible 
curve, transverse to the line at infinity.  Let $d$ be the degree of $L$.  
Then $L$ can be viewed as an (orientable) Riemann surface of genus 
$g=\binom{d-1}{2}$, with $d$ punctures corresponding to 
the points where $L$ meets the line at infinity.  
By construction, the vertex manifold $M_L$ is the 
boundary of a tubular neighborhood of $L$ inside $\C^2$, intersected 
with a ball centered at $0$ and containing all the points in $\mP$.  
As the normal bundle of $L$ is trivial, the vertex manifold $M_L$ is 
homeomorphic to the product of $S^1$ with a copy of $L$ from which 
several open disks (one for each point $P\in \P_L$) have been removed. 
It follows that $e(M_L)=0$, and so axioms \ref{x2} and \ref{x3} hold for 
the vertex manifold $M_L$.

Likewise, each intersection point $P\in \mP$ is a singularity of type ${\rm A}$, 
and so its singularity link is the Hopf link on $\abs{\mL_P}$ components.  
Consequently, the vertex manifold $M_P$ is the exterior of  this link, and 
thus homeomorphic to $S^1\times S_P$, where $S_P$ is a sphere with 
a number of disks removed (one for each $L\in \L_P$).  
The idea outlined in Example \ref{ex:hopf} shows that the 
Euler number $e(M_P)$ is equal to one.  Thus, axioms \ref{x2} 
and \ref{x4} hold for the vertex manifold $M_P$.  

Finally, as shown in the aforementioned references,  
the vertex manifolds are glued along tori via flip maps, in 
a manner specified by the edges of the plumbing graph $\Gamma$. 
Thus, axiom \ref{x5} is verified, and this completes the proof.  
\end{proof}

We single out an immediate corollary, for future use. 

\begin{cor}
\label{cor:bdryarr}
Let $\A$ be an arrangement of lines in $\C^2$.  Then the boundary 
manifold $M_{\A}$ lies in $\mX$. 
\end{cor}

\section{Applying the combination theorem to boundary manifolds}
\label{sect:graph rfrp}

This section is devoted to proving Theorem \ref{thm:graph manifold} from the 
introduction, using Theorem \ref{thm:combination} as the main tool.  

\subsection{A Mayer--Vietoris sequence}
\label{subsect:MV}
In an arbitrary graph manifold $X$, even within the class $\mX$ we have defined 
previously, it is generally still not true that the inclusion $X_v\to X$ 
of a vertex manifold induces an injection 
$H_1(X_v,\bZ)\to H_1(X,\bZ)$.
We will now give conditions on the graph $\gam$ which will guarantee 
that the inclusion of a vertex manifold induces an injection on first homology, 
for graph manifolds within the class $\mX$.

First observe that the Mayer--Vietoris sequence for $X=X_{\gam}$ implies that 
\begin{equation}
\label{eq:mv h1x}
H_1(X,\bZ)\cong H_V\oplus\bZ^{b_1(\gam)},
\end{equation}
where $H_V$ is the image of the map induced on homology by 
the inclusion 
\begin{equation}
\label{eq:coprod xv}
\coprod_{v\in V(\gam)} X_v\to X.
\end{equation}

Recall we are assuming that each vertex manifold $X_v$ is homeomorphic to 
$S^1\times S_v$, where  $S_v$ is obtained by deleting a number of disjoint, 
open disks from a closed, orientable surface $S$.  

For each vertex $v$, we will write $B_v\subset H_1(X_v,\bZ)$ 
for the subgroup generated by the homology classes of the boundary 
components of $S_v$.  Specifically, we consider the inclusion 
$\bigcup_{i=1}^{m(v)} S^1_i\to S_v\subset X_v$ 
given by sending $S^1_i\to\partial D^2_i$, and we set $B_v$ to be the 
image of the induced map on first homology.
We then have a direct sum decomposition 
\begin{equation}
\label{eq:h1 sv}
H_1(S_v,\bZ)\cong B_v\oplus W_v, 
\end{equation}
where roughly $W_v$ is ``generated by the genus" of $S_v$.  
Note that when we assemble $X$, all the gluing maps 
are performed along boundary components of $X_v$.  The 
Mayer--Vietoris sequence implies then that 
\begin{equation}
\label{eq:hv sum}
H_V\cong \langle B_V,t_v\mid v\in V(\gam)\rangle\oplus 
\Big(\bigoplus_{v\in V(\gam)} W_v\Big),
\end{equation}
where $B_V$ is the total homology span of the boundary components 
of $\{S_v\}_{v\in V(\gam)}$ inside of $X$.  

Recall from \S\ref{subsec:class} that for each vertex $L\in \mL$, the 
surface $S_L$ has boundary curves $C_{L,P}$ indexed by $P\in \mP_L$, 
and some extra boundary curves $C^1_{L},\dots, C^{r(L)}_L$. Recall also that 
we denote the corresponding homology classes in $H_1(S_L,\Z)$ by $b_{L,P}$ 
and  $b^{i}_{L}$, respectively.
Likewise, for each vertex $P\in \mP$, the surface $S_P$ has boundary curves 
$C_{P,L}$ indexed by $L\in \mL_P$, and the  homology classes of these curves 
are denoted by $b_{P,L}$. With this notation, we have that 
\begin{equation}
\label{eq:bl}
B_L=\langle\{b_{L,P}\}_{P\in\mP_L}\cup\{b^1_{L},b^2_{L},\ldots, b^{r(L)}_{L}\}\rangle
\:\ \text{and} \:\ 
B_P=\langle\{b_{P,L}\}_{L\in\mL_P}\rangle.
\end{equation} 

Recall from \S\ref{subsec:framings} that, for each vertex $v$, we 
denote by $t_v$ the homology generator of the $S^1$ factor of $X_v$.  
For $v=L$, we will write $\Xi_L$ for the subgroup of $B_L$ generated by 
the classes $b^2_{L},\ldots, b^{r(L)}_{L}$, and we will set 
\begin{equation}
\label{eq:xi}
\Xi=\bigoplus_{L\in \mL}\Xi_L.
\end{equation} 
Observe that if $\deg L=1$ then $\Xi_L=0$.

With this notation, the group $B_V$ is the span of the images of 
$\{b_{L,P} \}$,  $\{b_{P,L}\}$, and $\Xi$, where $L$ and $P$ 
range over $\mL$ and $\mP$ respectively.  We note that it is immediate 
from the Mayer--Vietoris sequence that the subgroup $\Xi$ breaks off 
as a direct summand of $B_V$.

\begin{lem}
\label{lem:tlgens}
Let $X\in\mX$.  Then there is a finite index subgroup of the abelian 
group \[\langle B_V,t_v\mid_{v\in V(\gam)}\rangle/\Xi\] which is freely 
generated by the homology classes $\{t_L\}_{L\in\mL}$.
\end{lem}

\begin{proof}
The Mayer--Vietoris sequence for $X$ says that the image of 
$\langle B_V,t_v\mid_{v\in V(\gam)}\rangle$ is a quotient of 
\begin{equation}
\label{eq:tvbv}
\bigoplus_{v\in V(\gam)} \langle B_v, t_v\rangle,
\end{equation}
with relations given by the gluing maps $\psi_e\times\bar{\psi}_e$.
Let $L,K\in\mL_P$.  The gluing maps impose the relations 
\begin{equation}
\label{eq:tpbkp}
b_{L,P}=t_P=b_{K,P}.
\end{equation}

Similarly, let $P,Q\in\mP_L$.  The gluing maps impose the relations 
\begin{equation}
\label{eq:tlbql}
b_{P,L}=t_L=b_{Q,L}.
\end{equation}  
It follows that $\langle B_V,t_v\mid_{v\in V(\gam)}\rangle$ is in 
fact generated by $\{t_v\}_{v\in V(\gam)}$.  The only remaining relations 
come from the nonzero Euler numbers $e(P)=k_P$  of the vertex spaces 
$\{X_P\}_{P\in\mP}$.  Recall from \eqref{eq:Euler1} that these relations say that 
\begin{equation}
\label{eq:kptp}
k_P\cdot t_P=\sum_{L\in\mL_P} t_L.
\end{equation} 

It follows immediately that $\langle B_V,t_v\mid_{v\in V(\gam)}\rangle/\Xi$ 
has a finite index subgroup which is generated by $\{t_L\}_{L\in\mL}$, and 
that no further relations among these generators hold.
\end{proof}

\subsection{Girth and homological injectivity}
\label{subsect:girth}
Recall that the \emph{girth}\/ of a graph $\gam$ is the length of the shortest 
non-backtracking loop in $\gam$.

\begin{thm}
\label{thm:injcriterion}
Let $X\in\mX$, and suppose that the girth of the defining graph $\gam$ 
is at least six.  Then the inclusion $X_v\to X$ induces an injection 
$H_1(X_v,\bZ)\to H_1(X,\bZ)$.
\end{thm}

\begin{proof}
First, suppose that $v=P$.  We have that 
\begin{equation}
\label{eq:h1xp}
H_1(X_P,\bZ)=\Big\langle t_P,\{b_{P,L}\}_{L\in\mL_P}\mid k_P\cdot t_P
=\sum_{L\in\mL_P} b_{P,L}\Big\rangle.
\end{equation}
In $H_1(X,\bZ)$, we have that the image of the finite index subgroup of 
$H_1(X_P,\bZ)$ generated by $k_P\cdot t_P$ and by $\{b_{P,L}\}_{L\in\mL_P}$ 
is in fact generated by $\{t_L\}_{L\in\mL_P}$.  These latter elements generate a free 
abelian group of the same rank as $H_1(X_P,\bZ)$, by Lemma \ref{lem:tlgens}.  It follows 
that the inclusion $X_P\to X$ induces an injective map $H_1(X_P,\bZ)\to H_1(X,\bZ)$.

Now, suppose that $v=L$.  Consider the finite index subgroup of $H_1(X_L,\bZ)/\Xi_L$
generated by (the images of) $t_L$ and by $\{k_P\cdot b_{L,P}\}_{P\in\mP_L}$.  Since 
$b_{L,P}$ is identified with $t_P$, the image of 
$k_P\cdot b_{L,P}$ in $H_1(X,\bZ)$ is given by 
\begin{equation}
\label{eq:blp}
\beta_{L,P}:=\sum_{K\in\mL_P} t_K,
\end{equation}
as follows from the gluing relations in $X$.

It suffices to show that the classes $t_L$ and $\{\beta_{L,P}\}_{P\in\mP_L}$ 
are linearly independent in $H_1(X,\bZ)$.  This will establish the theorem, 
because the inclusion $X_L\to X$ then induces a map $H_1(X_L,\bZ)\to H_1(X,\bZ)$, 
whose image contains a free abelian subgroup whose rank is the same 
as that of $H_1(X_L,\bZ)$, so that this map must be injective.

For each class $\beta_{L,P}$, choose a vertex $K\in\mL_P\setminus \{L\}$.  
Such a vertex exists, because we assumed that 
the degree of $P$ is at least two.  
Now let $Q\in\mP_L\setminus \{P\}$.  Notice that $\lk(P)\cap\lk(Q)=\{L\}$, 
since the girth of $\gam$ is at least six.  Therefore, for each $P,Q\in\lk(L)$, 
we can find a vertex $K_P\in\mL_P\setminus \{L\}$ and $K_Q\in\mL_Q\setminus \{L\}$ 
such that $K_P\neq K_Q$ for $P\neq Q$.  Furthermore, in the expressions 
\begin{equation}
\label{eq:sum tk}
\beta_{L,P}=\sum_{K\in\mL_P} t_K
\quad\text{and}\quad 
\beta_{L,Q}=\sum_{H\in\mL_Q} t_H,
\end{equation} 
we have that 
\begin{equation}
\label{eq:tktl}
\{t_K\}_{K\in\mL_P}\cap\{t_H\}_{H\in\mL_Q}=\{t_L\},
\end{equation}
since $\mL_P\cap\mL_Q=\{L\}$. Thus, the generator $t_{K_P}$ occurs 
in the expression of $\beta_{L,P}$ and no other class $\beta_{L,Q}$ for 
$P\neq Q$.  Since 
\begin{equation}
\label{eq:bvt}
\langle B_V,t_v\mid_{v\in V(\gam)}\rangle/\Xi
\end{equation}
is virtually freely generated by $\{t_L\}_{L\in\mL}$, it follows immediately 
that the classes $t_L$ and $\{\beta_{L,P}\}_{P\in\mP_L}$ are linearly independent.
\end{proof}

\subsection{Primitive lattices}
\label{subsec:primitive}
Before proceeding, we need to establish a couple of lemmas. 
Let $A$ be a finitely generated abelian group, and let $B<A$ be a subgroup. 
 We say that $B$ is \emph{primitive}\/ if the inclusion $B\to A$ is a split injection.

\begin{lem}
\label{l:pushoutAmalgam}
Let $A,B,C$ be finitely generated, torsion-free abelian groups, and let $i_A\colon C\to A$ 
and $i_B\colon C\to B$ be injective maps such that $i_A(C)$ and $i_B(C)$ are primitive.  
Then the natural copies of $A$ and $B$ in the pushout 
\[
P= (A\oplus B)/(i_A(C)=i_B(C))
\] 
are primitive.  Furthermore, $P$ is torsion-free.
\end{lem}

\begin{proof}
Since $i_B(C)<B$ is primitive, we have that 
$B\cong i_B(C)\oplus B'$ for some complement $B'$.  
Similarly, $A\cong i_A(C)\oplus A'$, so that 
$A'\cap B'=\{0\}$ in $P$, and so that $i_A(C)$ and $i_B(C)$ are identified in 
$P$ via the inverses of $i_A$ and $i_B$, respectively. We therefore 
have an isomorphism $P/B'\cong A$.  Composing this isomorphism with 
the canonical projection $P\surj P/B'$, we obtain an epimorphism 
$P\surj A$ which splits the natural inclusion of $A$ into $P$.  
Switching the roles of $A$ and $B$, we have the conclusion of the lemma.  
It is clear that $P$ is torsion-free.
\end{proof}

\begin{lem}
\label{l:pushoutHNN}
Let $A$ be a finitely generated abelian group, let $B,C<A$ be subgroups with an 
isomorphism $\phi\colon B\to C$, and let $G_{\phi}$ be the abelian HNN extension 
of $A$ along $\phi$, i.e.,
\[
G_{\phi}\cong\langle A,t\mid [t,A]=1, C=\phi(B)\rangle.
\]  
Let $D<A$ be a primitive, torsion-free subgroup such that the inclusion $D\to A$ 
descends to an injection $D\to G_{\phi}$.  Then $D<G_{\phi}$ is primitive.
\end{lem}

\begin{proof}
The kernel of the canonical projection map $A\times \langle t\rangle 
\to G_{\phi}$ is  the group 
\begin{equation}
\label{eq:group}
K=\langle\{\phi(b)-b\mid b\in B\}\rangle.
\end{equation} 
Since the inclusion $D\to A$ projects to an inclusion $D\to G_{\phi}$, we have that 
$K\cap D=\{0\}$.  Thus, we have that $D$ and $K$ span a subgroup of $A$ isomorphic 
to $D\oplus K$.  We claim that $K$ can be extended to a complement $H$ for $D$ in $A$, 
so that the inclusion of $D$ into $G_{\phi}$ is split.

Let $T<A$ be the torsion subgroup.  We have that $D$ is still primitive in $A/T$.  
It is easy to check that the images of $K$ and $D$ in $A/T$ still have trivial intersection.  
Indeed, suppose $x\in D$ and $y\in K$ differ by a torsion element $t$ of order $n$, 
so that $x=y\cdot t$.  Then since $D$ is torsion-free, we have that 
\[
0\neq x^n= y^n\in K\cap D,
\] 
a contradiction.  Thus, we may find a map $A/T \to D$ which splits the inclusion of 
$D$ into $A$, and for which $K$ lies in the kernel.  This map factors through 
$G_{\phi}$, so that the inclusion $D\to G_{\phi}$ is primitive.
\end{proof}

\subsection{Promoting injectivity to split-injectivity}
\label{subsec:split int}
We now refine Theorem \ref{thm:injcriterion} slightly, which will allow us to prove that 
the class $\mX$ has certain desirable closure properties with respect to taking finite covers. 
Namely, we will now show that under the hypothesis that the girth of the defining 
graph $\gam$ is at least six, the inclusion $X_v\to X$ induces a split injection 
on the level of first integral homology.

We wish to show that under certain general conditions, the vertex groups $\{G_v\}$ 
and the edge groups $\{G_e\}$ in a graph of groups $G_{\gam}$ split on the level 
of homology.  That is to say, we will give some general conditions under which the 
inclusion $G_v\to G_{\gam}$ induces a split map $H_1(G_v,\bZ)\to H_1(G_{\gam},\bZ)$.  
We will generally assume that the groups $\{G_v\}$ and $\{G_e\}$ all include into 
$G_{\gam}$, and that all these groups are finitely generated.

\begin{thm}
\label{t:split}
Let $G_{\gam}$ be a graph of groups with vertex groups $\{G_v\}$ and edge 
groups $\{G_e\}$.  Suppose that:
\begin{enumerate}
\item \label{hyp1}
For each $v$, the inclusion $G_v\to G_{\gam}$ induces an injection 
$H_1(G_v,\bZ)\to H_1(G_{\gam},\bZ)$.
\item  \label{hyp2}
The group $H_1(G_v,\bZ)$ is torsion-free for each $v\in V$.
\item  \label{hyp3}
The inclusion $G_e \to G_v$ induces a split injection $H_1(G_e,Z) \to H_1(G_v,Z)$.
\end{enumerate}
Then the inclusion $G_v\to G_{\gam}$ induces a split injection 
$H_1(G_v,\bZ)\to H_1(G_{\gam},\bZ)$.
\end{thm}

\begin{proof}
It suffices to compute the abelianization of the group $G_{\gam}$, which is described 
as an iterated amalgamated product and HNN extension via the gluing data specified 
by the edge groups.

Let $T\subset\gam$ be a maximal tree, and let $G_T$ be the associated graph of groups.  
On the level of homology, an easy induction using the Mayer--Vietoris sequence and 
Lemma \ref{l:pushoutAmalgam} implies the conclusion for $H_1(G_v,\bZ)<H_1(G_T,\bZ)$.

The conclusion for $G_{\gam}$ now follows from Lemma \ref{l:pushoutHNN} and an 
easy induction on $|E(\gam)\setminus E(T)|$.
\end{proof}

\begin{cor}\label{cor:split injection}
Let $X\in\mX$ have defining graph $\gam$ of girth at least six. 
Then for each $v\in V(\gam)$, the inclusion $X_v\to X$ induces 
a split injection $H_1(X_v,\bZ)\to H_1(X,\bZ)$.
\end{cor}

\subsection{Propagating homological injectivity to finite covers}
\label{subsec:propagate}

Let $Z$ have the homotopy type of a finite CW-complex, let $p$ be a prime, and let 
$Z'\to Z$ be the finite cover classified by the natural map 
\begin{equation}
\label{eq:cl map}
\xymatrixcolsep{16pt}
\xymatrix{\pi_1(Z)\ar[r]& H_1(Z,\bZ/p\bZ)}.
\end{equation}
We say this cover is the \emph{$p$-congruence cover} of $Z$. The cover $Z'$ is the 
\emph{torsion-free $p$-congruence cover}\/ of $Z$ if instead we consider the 
natural map 
\begin{equation}
\label{eq:tf h1 z}
\xymatrixcolsep{16pt}
\xymatrix{\pi_1(Z)\ar[r]& (\TF H_1(Z,\bZ))\otimes \bZ/p\bZ}.
\end{equation}

Let $X$ is a graph of spaces with underlying graph $\gam$, and let $\gam'\to\gam$ be 
a finite $p$-cover classified by a surjective homomorphism $\pi_1(\gam)\to G$, where 
$G$ is a finite $p$-group. We say that $X'\to X$ is a \emph{girth-fixing $p$-cover}\/ 
if it is classified by a composition of (surjective) homomorphisms 
\begin{equation}
\label{eq:composition map}
\xymatrixcolsep{14pt}
\xymatrix{\pi_1(X)\ar[r]&\pi_1(\gam)\ar[r]& G},
\end{equation}
where the first map is induced by the collapsing map $\kappa$, and where the 
corresponding cover $\gam'$ of $\gam$ has girth at least six.

\begin{lem}
\label{lem:propagation}
Let $X\in\mX$ with underlying graph of girth at least six, let $X'\to X$ be the 
torsion-free $p$-congruence cover, and let $X''\to X'$ be a girth-fixing $p$-cover. 
Then $X',X''\in\mX$. Furthermore, the natural inclusion $X''_v\to X''$ of a vertex 
space induces a split injection on the level of first integral homology.
\end{lem}
\begin{proof}
For the first statement, we need only check that $X'$ and $X''$ satisfy membership 
criteria for $\mX$. Write $\gam'$ and $\gam''$ for the respective underlying graphs 
of the natural pulled back graph of spaces structure. The coloring of the graph $\gam$ 
pulls back to colorings of $\gam'$ and $\gam''$, and the degrees of vertices of 
$\gam'$ and $\gam''$ cannot decrease from those of $\gam$.

If $X_v$ is a vertex space of $X$, then each component $X_v'$ of the preimage of 
$X_v$ in $X'$ is simply the $p$-congruence cover of $X_v$, and similarly for $X''$, 
as follows from Corollary \ref{cor:split injection}. Thus, for each $v\in V(\gam)$, the 
vertex space $X_v$ will be a product of a circle with an orientable surface with boundary.

Let $v=L$. Then it is evident that the zero Euler number relation \eqref{eq:Euler0} 
for $X_L$ pulls back to a zero Euler number relation for $X'_L$, and that $X_L'$ 
will have boundary components which are boundary components of both $X'$ 
and of $X''$. 

Let $v=P$. Then since no boundary component of $X_P$ is a boundary 
component of $X$, the same will be true of $X'_P$. Furthermore, the 
nonzero Euler number relation  \eqref{eq:Euler1}  
simply replaces $k_P$ by a nonzero integer multiple.

The fact that the gluing maps are flips in $X$ immediately implies that the gluing maps 
are flips in $X'$ and $X''$, since the covers of the vertex spaces preserve the circle 
and surface directions. Thus, $X'$ and $X''$ lie in $\mX$.

The second claim of the lemma follows from Corollary \ref{cor:split injection}.
\end{proof}

\subsection{Edge groups are closed in the RFR$p$ topology}
\label{subsec:edges}
Let $X=X_v$ be a trivial circle bundle over a compact, orientable surface with boundary, 
and let $X_e\subset X$ be a boundary component. We identify $\pi_1(X)$ 
with $\Z\times F$, where $F$ is a finitely generated free group 
and where $\Z$ is generated by the circle direction. Furthermore,  
we identify $\pi_1(X_e)$ with a copy of $\Z^2$ inside of $\Z\times F$. 
We claim that this subgroup is always closed in the RFR$p$ topology on 
$\Z\times F$. This fact follows easily from Theorem \ref{thm:separable}, 
but we give a direct argument which is more elementary:

\begin{lem}
\label{l:edge group closed}
Let $\Z^2<\Z\times F$ be a maximal rank two abelian subgroup and let 
$p$ be a prime. Then $\Z^2$ is closed in the RFR$p$ topology on $\Z\times F$.
\end{lem}

\begin{proof}
We will write $t$ for a generator of the central copy of $\Z$ in $\Z\times F$. 
It is easy to check that if $\Z^2<\Z\times F$ is maximal then $t\in \Z^2$. 
Thus we may suppose that $\Z^2$ is generated by $t$ and by $x\in F$, 
where $\langle x\rangle$ is a maximal cyclic subgroup of $F$.

Let $G_1=\Z\times F$ and let $\{G_i\}_{i\geq 1}$ be the standard RFR$p$ filtration 
on $\Z\times F$. It suffices to show that if $g\notin \Z^2$ then the image of 
$g$ in $G_1/G_i$ does not coincide with the image of $\Z^2$ in $G_1/G_i$, for $i\gg1$.

We may suppose that $g=t^ny$, where $1\neq y\in F$ is not contained in $\langle x\rangle$. 
Note that $[x,g]\neq 1$, so that there is 
some $i$ such that the image of $[x,g]$ is nontrivial in $G/G_i$. But then the image of 
$g$ in $G/G_i$ does not commute with $x$ and therefore cannot lie in the image of 
the abelian group $\Z^2$.
\end{proof}

\subsection{Graph manifolds of type $\mX$ have the RFR$p$ property}
\label{subsec:endgame}

We are now ready to complete the proof of Theorem \ref{thm:graph manifold} 
from the introduction, stating that the fundamental group of a graph 
manifold which belongs to the class $\mX$ defined in \S\ref{subsec:3mfd} 
is RFR$p$, for all primes $p$. 

\begin{proof}[Proof of Theorem \ref{thm:graph manifold}]
We only need to verify the hypotheses of Theorem \ref{thm:combination}. 
Let $X\in \mX$, and let $\{X_i\}_{i\geq 1}$ be the standard RFR$p$ 
tower of $X$, as usual. We refine the tower $\{X_i\}_{i\geq 1}$ slightly to a tower 
$\{Y_i\}_{i\geq 1}$ by taking an intermediate girth-fixing $p$-cover at each stage. 
Namely, we first let $Y_1\to X_1$ be an arbitrary girth-fixing $p$-cover. In general, 
define $Z_{i+1}\to Y_i$ to be the usual torsion-free $p$-congruence cover of $Y_i$, 
and let $Y_{i+1}\to Z_{i+1}$ be an arbitrary girth-fixing $p$-cover. It is easy to 
check that for $j\gg i$, we have $\pi_1(X_j)<\pi_1(Y_i)$.

By Lemma \ref{lem:propagation}, we have that $Y_i\in\mX$ for each $i$, and for each 
vertex space $Y_{v,i}$ of $Y_i$, the inclusion map induces a split injection on the level 
of homology. But then it follows immediately that $\{Y_{v,i}\}_{i\geq 1}$ is equal to the 
usual RFR$p$ tower for $X_v$, where $X_v$ is a vertex space of $X$ covered by 
each level in the tower $\{Y_{v,i}\}_{i\geq 1}$. We then have that the RFR$p$ 
topology on $\pi_1(X)$ induces the RFR$p$ topology on $\pi_1(X_v)$, for each 
vertex subspace $X_v \subset X$.

Thus, we need only see that the edge spaces are closed in the RFR$p$ topology. 
This follows immediately from the topological description of the edge spaces and 
Lemma \ref{l:edge group closed}.
\end{proof}

\newcommand{\arxiv}[1]
{\texttt{\href{http://arxiv.org/abs/#1}{arXiv:#1}}}
\newcommand{\arxi}[1]
{\texttt{\href{http://arxiv.org/abs/#1}{arxiv:}}
\texttt{\href{http://arxiv.org/abs/#1}{#1}}}
\newcommand{\doi}[1]
{\texttt{\href{http://dx.doi.org/#1}{doi:#1}}}
\renewcommand{\MR}[1]
{\href{http://www.ams.org/mathscinet-getitem?mr=#1}{MR#1}}


\begin{thebibliography}{99}

\bibitem{agolrfrs} I.~Agol,
\emph{Criteria for virtual fibering},
J. Topology (2008), no.~2, 269--284.
\MR{2399130}

\bibitem{agolprivate} I.~Agol,
Private communication, 2012.

\bibitem{agolgm} I.~Agol. \emph{The virtual Haken conjecture}, 
with an appendix by I.~Agol, D.~Groves, and J.~Manning, 
Doc. Math. \textbf{18} (2013), 1045--1087.
\MR{3104553}

\bibitem{AF11} M.~Aschenbrenner, S.~Friedl, 
{\em Residual properties of graph manifold groups}, 
Topology Appl. \textbf{158} (2011), no.~10, 1179--1191. 
\MR{2796120}

\bibitem{AF13} M.~Aschenbrenner, S.~Friedl, 
{\em $3$-manifold groups are virtually residually $p$}, 
Mem. Amer. Math. Soc. \textbf{225} (2013), no.~1058. 
\MR{3100378}

\bibitem{AFW16} M.~Aschenbrenner, S.~Friedl, H. Wilton,
{\em $3$-Manifold groups}, EMS Series of Lectures in Mathematics, 
vol. 20, Eur. Math. Soc., 2015. 
\MR{3444187}

\bibitem{BS76}  G.~Baumslag, R.~Strebel,
{\em Some finitely generated, infinitely related metabelian groups 
with trivial multiplicator}, J.~Algebra \textbf{40} (1976), no.~1, 46--62. 
\MR{0422432}

\bibitem{bns} R.~Bieri, W.~Neumann, R.~Strebel, 
{\em A geometric invariant of discrete groups}, 
Invent. Math. \textbf{90} (1987), no.~3, 451--477. 
\MR{0914846} 

\bibitem{BK} E.~Brieskorn, H.~Kn\"{o}rrer,
{\em Plane algebraic curves}, Birkh\"{a}user Verlag, Basel, 1986.
\MR{0886476} 

\bibitem{Br} K.~S.~Brown,
{\em Cohomology of groups},
Grad. Texts in Math., vol. 87, Springer-Verlag, New York--Berlin, 1982.
\MR{0672956}

\bibitem{CS08} D.~Cohen, A.I.~Suciu,
{\em The boundary manifold of a complex line arrangement}, 
Geometry \& Topology Monographs \textbf{13} (2008), 105--146.
\MR{2508203}

\bibitem{Dehornoyetal} P.~Dehornoy, I.~Dynnikov, D.~Rolfsen, B.~Wiest,
{\em Ordering braids}, Mathematical Surveys and Monographs, vol.~148, 
Amer. Math.Society, Providence, RI, 2008. 
\MR{2463428}  

\bibitem{DeroinNavasRivas} B.~Deroin, A.~Navas, C.~Rivas,
{\em Groups, orders, and dynamics}, in press,
Amer. Math. Society, Providence, RI.
\arxiv{1408.5805}

\bibitem{DP}  A.~Di~Pasquale, 
 {\em Arrangements of complex plane algebraic curves, 
their complements and their links}, Ph.D. thesis, University 
of Melbourne, 1999. 
\url{http://www.math.columbia.edu/~neumann/preprints/Thesis_Angelo_Di_Pasquale.pdf}

\bibitem{DoigHorn}  M.~Doig, P.~Horn, 
{\em On the intersection ring of graph manifolds}, 
Trans. Amer. Math. Soc. \textbf{369} (2017), no. 2, 1185--1203. 
\MR{3572270}

\bibitem{Du83} A.~Durfee, 
{\em{Neighborhoods of algebraic sets}}, 
Trans. Amer. Math. Soc. \textbf{276} (1983), no.~2, 517--530. 
\MR{0688959}

\bibitem{EN} D.~Eisenbud, W.~Neumann, 
{\em Three-dimensional link theory and invariants of plane curve
singularities}, Annals of Math. Studies, vol.~110, Princeton 
University Press, Princeton, NJ, 1985.
\MR{0817982}

\bibitem{friedltillman} S.~Friedl, S.~Tillman,
{\em Two-generator one-relator groups and marked polytopes},
to appear in Ann. Inst. Fourier (Grenoble). 
\arxiv{1501.03489} 

\bibitem{HP15} M.~Hagen, P.~Przytycki,
{\em Cocompactly cubulated graph manifolds},
Israel J. Math. \textbf{207} (2015), no.~1, 377--394. 
\MR{3358051}

\bibitem{hamilton} E.~Hamilton, 
{\em Abelian subgroup separability of Haken $3$-manifolds and 
closed hyperbolic $n$-orbifolds}, Proc. London Math. Soc. 
\textbf{83} (2001), no.~3, 626--646.
\MR{1851085}

\bibitem{Hi97} E.~Hironaka,
{\em Alexander stratifications of character varieties}, Annales 
de l'Institut Fourier (Grenoble) \textbf{47} (1997), no.~2, 555--583.
\MR{1450425}

\bibitem{Hi01}  E.~Hironaka,
{\em Boundary manifolds of line arrangements}, 
Math. Annalen \textbf{319} (2001), no.~1, 17--32. 
\MR{1812817}

\bibitem{HW99} T.~Hsu, D.~Wise, 
{\em On linear and residual properties of graph products}, 
Michigan Math. J. \textbf{46} (1999), no.~2, 251--259. 
\MR{1704150}

\bibitem{JY98} T.~Jiang, S.S.-T.~Yau,  
{\em Intersection lattices and topological structures of 
complements of arrangements in $\CP^2$}, Ann. Scuola 
Norm. Sup. Pisa Cl. Sci. \textbf{26} (1998), no.~2, 357--381.  
\MR{1631597}

\bibitem{IMSProc} T.~Koberda, 
{\em Ping-pong lemmas with applications to geometry and topology}, in: 
Geometry, topology and dynamics of character varieties, 139--158, 
Lect. Notes Ser. Inst. Math. Sci. Natl. Univ. Singap., vol.~23, 
World Sci. Publ., Hackensack, NJ, 2012.
\MR{2987617} 

\bibitem{Ko12} T.~Koberda, 
{\em Right-angled {A}rtin groups and a generalized isomorphism 
problem for finitely generated subgroups of mapping class groups}, 
Geom. Funct. Anal. \textbf{22} (2012), no.~6, 1541--1590. 
\MR{3000498}

\bibitem{Ko13} T.~Koberda, 
{\em Residual properties of fibered and hyperbolic $3$-manifolds}, 
Topology Appl. \textbf{160} (2013), no.~7, 875--886. 
\MR{3037878}
 
\bibitem{Ko14} T.~Koberda, 
{\em Alexander varieties and largeness of finitely presented groups}, 
J. Homotopy Relat. Struct. \textbf{9} (2014), no.~2, 513--531. 
\MR{3258692}

\bibitem{Liu13} Y.~Liu,
{\em Virtual cubulation of nonpositively curved graph manifolds},
J. Topol. \textbf{6} (2013), no.~4, 793--822.
\MR{3145140}

\bibitem{LW} J.~Luecke, Y.-Q. Wu, 
{\em Relative Euler number and finite covers of graph manifolds}, 
in: Geometric topology (Athens, GA, 1993), 80--103, 
AMS/IP Stud. Adv. Math., vol.~2.1, Amer. Math. Soc., 
Providence, RI, 1997. 
\MR{1470722}

\bibitem{LyndonSchupp} R.~Lyndon, P.~Schupp,
{\em Combinatorial group theory},
Reprint of the 1977 edition, Classics in Mathematics, 
Springer-Verlag, Berlin, 2001.
\MR{1812024}

\bibitem{Neu} W.~Neumann, 
{\em A calculus for plumbing applied to the topology of complex surface 
singularities and degenerating complex curves}, 
Trans. Amer. Math. Soc. \textbf{268} (1981), no.~2, 299--344. 
\MR{0632532}

\bibitem{PS10} S.~Papadima, A.I.~Suciu,
{\em Bieri--{N}eumann--{S}trebel--{R}enz invariants and 
homology jumping loci}, Proc.~London Math.~Soc. 
\textbf{100} (2010), no.~3, 795--834.  
\MR{2640291} 

\bibitem{PS18} S.~Papadima, A.I.~Suciu, 
{\em Infinitesimal finiteness obstructions}, J. London Math. Soc. 
(2018), \doi{10.1112/jlms.12169}.

\bibitem{Perelman1}  G.~Perelman,
{\em Ricci flow with surgery on three-manifolds}, 
\arxiv{math.DG/0303109}

\bibitem{Perelman2}  G.~Perelman,
{\em Finite extinction time for the solutions to the {R}icci 
flow on certain three-manifolds}, 
\arxiv{math.DG/0307245}

\bibitem{PW14} P.~Przytycki, D.~Wise, 
{\em Graph manifolds with boundary are virtually special},  
J. Topol. \textbf{7} (2014), no.~2, 419--435. 
\MR{3217626}

\bibitem{Rhemtulla1973} A.H.~Rhemtulla,
{\em Residually $F_p$-groups, for many primes $p$, are orderable}, 
Proc. Amer. Math. Soc. \textbf{41} (1973), no.~1, 31--33.
\MR{0332615}

\bibitem{Schroeder}  V.~Schroeder,
{\em Rigidity of nonpositively curved graphmanifolds}, 
Math. Ann. \textbf{274} (1986), no.~1, 19--26. 
\MR{0834102} 

\bibitem{Scottgeom} G.P.~Scott, 
{\em The geometries of\/ {$3$}-manifolds}, Bull. London Math. Soc. 
\textbf{15} (1983), no.~5, 401--487. 
\MR{705527}

\bibitem{StrebelPreprint} R.~Strebel,
{\em Metabelianisations of finitely presented groups}. 
\arxiv{1810.00176}

\bibitem{Su-imrn} A.I.~Suciu,
{\em Characteristic varieties and Betti numbers of free
abelian covers}, Int. Math. Res. Notices \textbf{2014} 
(2014), no. 4, 1063--1124.
\MR{3168402}

\bibitem{ThurstonBook} W.~Thurston, 
{\em Three-dimensional geometry and topology}, Vol.~1, 
Edited by Silvio Levy, Princeton Mathematical Series, 35, 
Princeton University Press, Princeton, NJ, 1997. 
\MR{1435975}

\bibitem{WZ10} H.~Wilton and P.~Zalesskii, 
{\em Profinite properties of graph manifolds},  
Geom. Dedicata \textbf{147} (2010), no.~1, 29--45. 
\MR{2660565}

\bibitem{WisePreprint} D.T.~Wise,
{\em The structure of groups with a quasiconvex hierarchy}, preprint, 2011.

\bibitem{WiseRAAGs} D.T.~Wise,
{\em From riches to raags: $3$-manifolds, right-angled Artin groups, and 
cubical geometry}, Amer. Math. Soc., Providence, RI, 2012.
\MR{2986461}

\end{thebibliography}
\end{document}